\documentclass[11pt]{amsart}

\usepackage{amsmath, amsthm, amssymb}
\usepackage{amsaddr}
\usepackage{verbatim,enumerate, bbm,bigints}
\usepackage{natbib,graphicx,fleqn}
\usepackage[colorlinks,linkcolor=blue,
			citecolor=blue,urlcolor=blue]{hyperref}
\allowdisplaybreaks[2]

\numberwithin{equation}{section}
\newtheorem{thm}{Theorem}

\newtheorem{lem}{Lemma}
\newtheorem{prop}{Proposition}
\theoremstyle{remark}

\theoremstyle{definition}
\newtheorem{defn}{Definition}

\newcommand{\1}[1]{\mathbbm{1}\!\left[#1\right]}
\newcommand{\E}[1]{\operatorname{E}\!\left[#1\right]}

\newcommand{\tr}[1]{\operatorname{tr}\!\left(#1\right)}

\newcommand{\diag}[1]{\text{diag}\!\left(#1\right)}

\begin{document}

\title{Noise Estimation in the Spiked Covariance Model}

\author[D. Ch\'etelat and M.T. Wells]{Didier Ch\'etelat and Martin T. Wells}
\address{Department of Statistical Science, Cornell University}
\email{dc623@cornell.edu}
\email{mtw1@cornell.edu}

\thanks{This research was partially supported by NSF Grants DMS-1208488 and CCF-0808864.}

\begin{abstract}
The problem of estimating a spiked covariance matrix in high dimensions under Frobenius loss, and the parallel problem of estimating the noise in spiked PCA is investigated. We propose an  estimator of the noise parameter by minimizing an unbiased estimator of the invariant Frobenius risk using calculus of variations. The resulting estimator is shown, using random matrix theory, to be strongly consistent and essentially asymptotically normal and minimax for the noise estimation problem. We apply the construction to construct a robust spiked covariance matrix estimator with consistent eigenvalues.
\end{abstract}



\maketitle

\section{Introduction}\label{sec:I}

The estimation of covariance matrices in a high dimensional framework has seen a surge of interest in the past years. The natural estimator, the sample covariance matrix, is well known to be inadequate in this context. The problem has been well studied under many sparsity scenarios: for example, zeros in the coordinates of the matrix \citep{Bickel08b,ElKaroui08b,Rothman09,Cai11b} or its inverse \citep{MeinshausenBuhlmann06,Friedman08,Cai11a,
Ravikumar11,Rothman08}, bandedness \citep{Bickel08a,Bien14} and many others. This paper will focus on the spiked model, first introducted by \cite{Johnstone01}.

In the spiked model, the $p\times p$ covariance matrix $\Sigma$ has distinct eigenvalues $\gamma_1+\sigma^2>...>\gamma_\rho+\sigma^2$, and a smallest eigenvalue $\sigma^2$ of multiplicity $p-\rho$. It often provides good approximations in low and high dimensional settings, with small $\rho$ being seen as a form of low rank sparsity in the data. It is also of substantial theoretical interest, being one of the few non-trivial settings in which random matrix theory has been extensively studied.

A related problem is principal components analysis. In PCA, one estimates eigenvectors associated with large eigenvalues of $\Sigma$, and perform dimension reduction using a truncated spectral decomposition. A traditional problem with the technique is that the number of eigenvectors to retain is not clear. However, if the true covariance matrix $\Sigma$ is spiked, it is natural to associate its spiked rank $\rho$ with the ideal number of eigenvectors to select, recasting the selection of the number of components as a rigorous statistical estimation problem.

Successful high-dimensional PCA usually requires good estimation of $\sigma^2$ (see e.g. \cite{Johnstone09}), a problem we will refer to as noise estimation. Although distinct from estimation of the covariance matrix itself, there is a context in which these two problems, estimation of $\Sigma$ and $\sigma^2$, are analogous.

This context is as follows. Asymptotics are high-dimensional in the sense that $p$ tends to infinity with the sample size $n$; for mathematical convenience we focus on the regime where the ratio $p/n$ tends to a strictly positive constant as $n\rightarrow\infty$. The noise estimation problem is to estimate $\sigma^2$ under, say, absolute error loss $L(\hat\sigma^2,\sigma^2)=|\hat\sigma^2-\sigma^2|$, while the covariance problem is to estimate the spiked $\Sigma$ under the Frobenius loss $L_F(\hat\Sigma,\Sigma)=\|\hat\Sigma-\Sigma\|_F^2/p$ using a spiked estimator. This normalization is natural in this setting, since under normality the risk $\E{L_F(S,\Sigma)}$ of the sample covariance matrix $S$ tends to a strictly positive constant.

Then, in essence, all that really matters in the covariance estimation problem is estimation of the noise level. Indeed, consider two spiked estimators $\hat\Sigma_i=\hat\Gamma_i+\hat\sigma^2I$, $i=1,2$ with asymptotically finite spiked parts $\hat\Gamma_i$, which we take to mean that their ranks $\hat\rho_i=\text{rk}(\hat\Gamma_i)$ and largest eigenvalues $\lambda_1(\hat\Gamma_i)$ are asymptotically finite. Then
\begin{align}
\frac{\|\hat\Sigma_1-\hat\Sigma_2\|_F^2}p\leq\frac{\hat\rho_1+ \hat\rho_2}p
\Big[\lambda_1(\hat\Gamma_1)+\lambda_1(\hat\Gamma_2)\Big]^2
\underset{n\rightarrow\infty}{\longrightarrow}0
\qquad
\text{ a.s.}
\label{eq:I-dual}
\end{align}
This means we can interpret the two problems as asymptotically analogous in practice. This reasoning is short of being a formal result of equivalence, but will serve as a guiding principle.

We propose a solution to these parallel problems as follows. We first restrict ourselves to orthogonally invariant estimators of the spiked form; this large class can be thought as performing spiked corrections of the eigenvalues of the sample covariance matrix. For this class, there exists an unbiased risk estimator (URE) in the closely related invariant loss $L_H(\hat\Sigma,\Sigma)=\|\hat\Sigma\Sigma^{-1}-I\|_F^2/p$, which we will refer to as the Haff loss. We propose to find an optimal choice of noise estimator by minimizing this URE using calculus of variations. This approach is close in spirit to the work of \cite{Stein75,Stein86}, where he considers a loss based on a normal log-likelihood, although it is not specifically high-dimensional. It is also close to the Bayesian approach of \cite{Haff91}. More generally, the idea to correct the eigenvalues of the sample covariance matrix is also found in previous work by  \cite{LedoitWolf04}, \cite{ElKaroui08a}, \cite{LedoitWolf12} and \cite{Donoho14}.

The URE of the covariance estimator depends on first and second derivatives of the noise estimator, so directly minimizing the risk would yield an estimator that depends on the truth. It however happily turns out that the ``dominant'' part of this URE does not depend on the derivatives. It is therefore possible to obtain, in closed form, an estimator optimal for the dominant part of the URE.

We prove that our proposed estimator is well-behaved; for example, it is strongly consistent for $\sigma^2$, even if the chosen estimators of $\gamma_k$ and $\rho$ are not. It is moreover essentially asymptotically normal of rate $n$, and we prove that this is the optimal minimax rate for the noise estimation problem. To illustrate concretely why this approach is interesting, we use it to construct a robust spiked covariance estimator. It seems to never perform worse than $S$ in general, even in worst-case scenarios; while it performs remarkably well in spiked settings, and we show its eigenvalues are consistent.

We reiterate that in contrast with much work in high dimensional covariance estimation, we do not work with a sparsity assumption that many components of $\Sigma$ or $\Sigma^{-1}$ are zero. However, one can perfectly think of a spiked structure as a form of sparsity in itself, with $\rho$ as sparsity parameter, which fits within the generally accepted principle that improved estimation in high dimensions is difficult unless some form of sparsity holds with the truth. The fact that we can construct an estimator that can exploit that structure when present, yet be robust to the assumption is encouraging.

The article is divided as follows. The regularity conditions, construction of the unbiased risk estimator and construction of the noise estimator is in Section \ref{sec:C}. Investigations of properties of the noise estimator is done in Section \ref{sec:P}. The example construction and simulations are in Section \ref{sec:A}. After some comments in Section \ref{sec:D}, we cover the proofs of the claims in Section \ref{sec:L}.

\begin{paragraph}{Notation}
The following notation will be used throughout. 
We write $H_p(\mathbb{R})$ for the simplex $\left\{x\in\mathbb{R}^p\,\big\vert\,x_1>...>x_p>0\right\}$. The real $p$-dimensional orthogonal group is denoted $O_p(\mathbb{R})$. The Frobenius norm of a matrix $A$ is the sum of its squared eigenvalues, denoted $\|A\|_F=\tr{A^2}^{1/2}$, while the spectral norm is its largest singular value, denoted $\|A\|_2=\sigma_{\max}(A)$. The notation $\text{d}_\text{TV}(\mu_1,\mu_2)$ stands for the total variation distance between two probability measures $\mu_1$, $\mu_2$ on an underlying measurable space $(\Omega,\mathcal{B})$, which equals $\sup_{A\in\mathcal{B}}|\mu_1(A)-\mu_2(A)|$. The $p$-dimensional Wishart distribution with $n$ degrees of freedom and covariance matrix $\Sigma$ is written $W_p(n,\Sigma)$.
\end{paragraph}
\section{Construction}\label{sec:C}

We work in the following setting. Assume the data is an i.i.d. sample $X_1,...,X_n\sim\text{N}_p(0,\Sigma_p)$, with $n\geq p$ and $\Sigma_p>0$. For such a sample, one can stack the data into a matrix $X=(X_1',...,X_n')$ and let $S=X'X/n=OLO'$, $L=\diag{l_1,...,l_p}$ be the decreasing spectral decomposition of the sample covariance matrix, with $l_1>...>l_p>0$ its ordered eigenvalues. The random matrix $S$, which is distributed as a scaled Wishart $n^{-1}\text{W}_p(n,\Sigma_p)$, serves as a naive estimator of $\Sigma$ upon which we wish to improve. The normality and restriction to $n\geq p$ are necessary for the construction of the unbiased risk estimator that will follow; extensions will be discussed in Section \ref{sec:D}.

As mentioned in the introduction, to adequately discuss high-dimensional behavior, we will also let this setting grow in complexity. We focus our attention on full-rank linear regimes, where a sequence of positive-definite covariance matrices of growing dimension $\Sigma_1,\Sigma_2,\Sigma_3,...$ is fixed; and $p=p_n$, as a function of the sample size, grows in the sense that $p_n/n\rightarrow c$ for some $c\in(0,1)$. It will then be assumed that for every $(n ,p_n)$, some i.i.d. sample $X_1,...,X_n\sim\text{N}_{p_n}(0,\Sigma_{p_n})$ will be available and a corresponding sample covariance matrix $S$ constructed. 

For such settings, the sequence $\{\Sigma_p\}$ is completely arbitrary beyond the requirement that each member be positive-definite. Of particular interest to us is the case where the covariance matrices form a {\it spiked sequence}, which we define as follows.
\begin{defn}
A sequence of covariance matrices $\{\Sigma_p\}$ is spiked if there exists a collection $\gamma_1>...>\gamma_\rho>0$ of size $\rho\geq0$ and a $\sigma^2>0$ such that for any $p$, $\Sigma_p=\diag{\gamma,0}+\sigma^2I_p$, where $\gamma=(\gamma_1,...,\gamma_\rho)$.
\end{defn}
When discussing asymptotics, we will sometimes need that the spiked eigenvalues $\gamma_1,...,\gamma_\rho$ be sufficiently large with respect to the noise for efficient estimation to be possible. In practice, this will mean requiring that   $\gamma_\rho/\sigma^2>\sqrt{c}$, for $c$ the asymptotic $p_n/n$ ratio. The importance of this supercriticality condition for spiked eigenvalue estimation was first remarked by \cite{Baik05} before being extended to the setting we are considering by  \cite{BaikSilverstein06}, \cite{Paul07} and \cite{Nadler08}. They showed that for $1\leq k\leq\rho$, the eigenvalues of the sample covariance matrix satisfy
\begin{align}
l_k\xrightarrow[n\rightarrow\infty]{\text{a.s.}}
\begin{cases} \left(1+c\sigma^2\frac{\gamma_k+\sigma^2}{\gamma_k}\right)[\gamma_k+\sigma^2]
&\text{ if } \gamma_k>\sqrt{c}\sigma^2
\\
(1+\sqrt{c})^2\sigma^2
&\text{ if } \gamma_k\leq\sqrt{c}\sigma^2
\end{cases},
\label{eq:l-limit}
\end{align}
with the $l_{\rho+1},...,l_p$ asymptotically distributed like a scaled Mar{\v c}enko-Pastur $\sigma^2\text{MP}(c)$ distribution. Therefore, the asymptotic spectrum of $S$ do not contain any information about those $\gamma_k$ below the critical threshold $\sqrt{c}\sigma^2$. But since their estimation is mostly tangential to our goals, supercriticality will not always be necessary, and we will make it clear when it will be.

Let us now turn our attention to the task at hand. The parallel problems we wish to solve are
\begin{enumerate}[(i)]
\item\label{eq:prob-noise} the estimation of $\sigma^2$ under the absolute error loss $L(\hat\sigma^2,\sigma^2)=|\hat\sigma^2-\sigma^2|$;
\item\label{eq:prob-cov} the estimation of $\Sigma$ under the Frobenius loss $L_F(\hat\Sigma,\Sigma)=\|\hat\Sigma-\Sigma\|_F^2/p$ using spiked estimators.
\end{enumerate}
Under spikedness, these two problems are parallel to each other in the sense of (\ref{eq:I-dual}). The approach we take begins with aspect (\ref{eq:prob-cov}) - we seek a good covariance estimator $\hat\Sigma$ in spiked form 
 $\hat\Gamma+\hat\sigma^2I$, with $\hat\rho=\text{rk}(\hat\Gamma)$ small with respect to $p$, which we interpret as $\hat\rho$ a.s. tending to a finite constant. By appealing again to (\ref{eq:I-dual}), it is clear that for the Frobenius loss, the spiked part $\hat\Gamma$ is asymptotically dominated by noise estimation. We therefore might as well {\it choose} $\hat\Gamma$ based on convenience: for example, we can pick one with consistent eigenvalue estimators, or some other property. A specific choice will be considered in Section \ref{sec:A}. Alternatively, the recent results of \cite{Donoho14} could provide an attractive choice based on consideration of single eigenvalue corrections, and we comment on this further in Section \ref{sec:D}. Once a choice of $\hat\Gamma$ is made, we can then look for an optimal $\hat\sigma^2$, which would simultaneously solve aspects (\ref{eq:prob-noise}) and (\ref{eq:prob-cov}) of the problem,  while being asymptotically independent of our choice of $\hat\Gamma$.

Being quite free in selecting the spiked part, let us focus on mathematically convenient possibilities. A first restriction is to take $\hat\Gamma$ orthogonally invariant -- that is, of the form $\hat\Gamma=O\diag{\hat\gamma,0}O'$ for some estimators $\hat\gamma_1>...>\hat\gamma_{\hat\rho}>0$ and $O\in O_p(\mathbb{R})$ the matrix of ordered eigenvectors of $S$. With this choice, our estimators $\hat\Sigma$ can be thought as performing spiked corrections on the sample covariance matrix $S$.

A second restriction will be necessary. At this point in our discussion the spiked, rank and noise estimators $\hat\gamma$, $\hat\rho$, $\hat\sigma^2$ have been essentially arbitrary. This is too general for the construction of the URE that will follow, so we must to restrict ourselves to sufficiently regular estimators. The regularity conditions come in two flavors, weak and strong, and are statements of integrability; these conditions simply guarantee that expected values appearing in the construction of the URE are convergent. Combining the invariance and regularity  restrictions, we define the following.

\begin{defn}\label{defn:WRC} A spiked eigenvalue estimator $\hat\Gamma$ satisfies the weak regularity conditions if it is of the form
$\hat\Gamma=O\diag{\hat\gamma}O'$ for $S=OLO'$ and satisfies the following. Let $\hat\rho=\text{rk}(\hat\Gamma)$. For each $1\leq k\leq p$, $\hat\gamma_k$ are a.s. $C^2(H_p(\mathbb{R});\mathbb{R})$ functions of $l_1,...,l_p$ with boundary cases $\1{\hat\rho<k}\hat\gamma_k=0$ and $\1{\hat\rho=p}\hat\gamma_k=\1{\hat\rho=p}l_k$ for which both expectations
\begin{align*}&
\E{\left|\frac{\hat\gamma_k}{l_k}\right|^{9(1+\epsilon)}}
\text{ and }
\E{\left|\frac{\partial \hat\gamma_k}{\partial l_k}\right|^{4.5}}
\end{align*}
are finite for some $\epsilon>0$. Similarly, a noise estimator $\hat\sigma^2$ satisfies the weak regularity conditions for a weak spiked eigenvalue estimator  $\hat\Gamma$ if it is a $C^2(H_p(\mathbb{R});\mathbb{R})$ function based on $l_1,...,l_p$ such that for each $1\leq k\leq p$,
\begin{align*}&
\E{\left|\frac{\hat\sigma^2}{l_k}\right|^{9(1+\epsilon)}}\!\!,\,
\E{\left|\frac{\partial \hat\sigma^2}{\partial l_k}\right|^{4.5}}
\text{ and }
\E{\vphantom{\Bigg\vert}\left|\hat\gamma_{k}+\hat\sigma^2\right|\left|\frac{\partial^2\hat\gamma_{k}}{\partial l_k^2}+\frac{\partial^2\hat\sigma^2}{\partial l_k^2}\right|}
\end{align*}
are all finite for some $\epsilon>0$.
\end{defn}

The previous conditions assert integrability of quantities associated with the estimators for a given $p$, and are dimension dependent. In contrast, the following conditions assert similar integrability as $p$ grows.  To make the dependence explicit, we superscript the dimension. 

\begin{defn}\label{defn:SRC} A spiked eigenvalue estimator $\hat\Gamma^p$ satisfies the strong regularity conditions if it satisfies the weak regularity conditions for each $p>0$, and moreover
{\setlength{\mathindent}{5pt}\begin{align*}&
\E{\sup_{p>0}\max_{\substack{1\leq k\leq p}}\left|\frac{\hat\gamma^p_k}{l^p_k}\right|^{9(1+\epsilon)}}\!\!,\,
\E{\sup_{p>0}\max_{\substack{1\leq k\leq p}}\left|\frac{\partial \hat\gamma^p_k}{\partial l^p_k}\right|^{4.5}}\!\!,\,
\E{\sup_{p>0}\max_{\substack{1\leq k\leq p}}\left|\hat\gamma_k^p\frac{\partial^2 \hat\gamma^p_k}{\partial {l^p_k}^2}\right|}\!\!,
\\&
\E{\sup_{p>0}\max_{\substack{1\leq k\neq b\leq \hat\rho}}\left|\frac{\hat\gamma^p_k-\hat\gamma^p_b}{l^p_k-l^p_b}\right|^{2}}\!\!,\,
\E{\sup_{p>0}\max_{\substack{1\leq k\neq b\\ \neq e\leq\hat\rho}}\left|\frac{l^p_k}{l^p_k-l^p_b}\right|^{2}\left|\frac{\hat\gamma^p_k-\hat\gamma^p_e}{l^p_k-l^p_e}-\frac{\hat\gamma^p_r-\hat\gamma^p_e}{l^p_b-l^p_e}\right|^{2}}\!\!
\\&\qquad\text{ and }
\E{\sup_{p>0}\max_{\substack{1\leq k\neq b\leq\hat\rho<e\leq p}}\left|\frac{l^p_k}{l^p_k-l^p_b}\right|^{2}\left|\frac{\hat\gamma^p_k}{l^p_k-l^p_e}-\frac{\hat\gamma^p_b}{l^p_b-l^p_e}\right|^{2}}
\end{align*}}%
are all finite for some $\epsilon>0$. Similarly, a noise estimator $\hat\sigma^2$ satisfies the strong regularity conditions for a strong spiked eigenvalue estimator  $\hat\Gamma$ if it satisfies the weak, and the following holds:
{\setlength{\mathindent}{5pt}\begin{align*}&
\E{\sup_{p>0}\max_{0\leq k\leq p}\left|\frac{\hat\sigma^{2p}}{l^p_k}\right|^{9(1+\epsilon)}}\!\!,\,
\E{\sup_{p>0}\max_{0\leq k\leq p}\left|\frac{\partial \hat\sigma^{2p}}{\partial l^p_k}\right|^{4.5}}\!\!
\\&\qquad\text{ and }
\E{\sup_{p>0}\max_{1\leq k\leq p}\vphantom{\Bigg\vert}\left|\hat\gamma^p_{k}+\hat\sigma^{2p}\right|\left|\frac{\partial^2\hat\gamma^p_{k}}{\partial l_k^{p2}}+\frac{\partial^2\hat\sigma^{2p}}{\partial l_k^{p2}}\right|}
\end{align*}}%
are all finite for some $\epsilon>0$.
\end{defn}

Careful inspection of the proofs reveal that regularity conditions in this spirit are inevitable; however, we emphasize that by no means we believe those precise conditions to be necessary, merely sufficient. In any case, with these conditions in hand we can formally define the classes of estimators in which we look for an optimal $\hat\sigma^2$.

\begin{defn} For $\hat\Gamma$ a weak (strong) spiked eigenvalue estimator, the associated weak (strong) class of spiked corrections to the sample covariance matrix is
{\setlength{\mathindent}{5pt}\begin{align*}&
V_{p}(\hat\Gamma)\!=\!\left\{\hat\Gamma\!+\!\hat\sigma^2I_p
\,\bigg\vert\,\hat\sigma^2\text{ is }\hat\Gamma\text{-weak}\right\}
\,\text{and}\;\;
\bar V_p(\hat\Gamma)\!=\!\left\{\hat\Gamma\!+\!\hat\sigma^2I_p
\,\bigg\vert\,\hat\sigma^2\text{ is }\hat\Gamma\text{-strong}\right\}\!.
\end{align*}}%
\end{defn}

We would like to find an optimal estimator over these two classes. Recall we are evaluating performance in Frobenius loss $L_F(\hat\Sigma,\Sigma)=\|\hat\Sigma-\Sigma\|_F^2/p$. Although natural and common within the literature, we find it more convenient to move to the closely related ``invariant'' loss
\begin{align*}
L_H(\hat\Sigma,\Sigma)=\frac{\|\hat\Sigma\Sigma^{-1}-I\|_F^2}{p}.
\end{align*}
This loss was, up to the high-dimensional $p^{-1}$ normalization, mentioned by \cite{JamesStein61} early but first thoroughly investigated by \cite{Haff77}, and we will refer to it as Haff's loss. A modification of the argument behind \eqref{eq:I-dual} shows that estimation of a spiked covariance matrix under this loss can also be thought as a noise estimation problem, just like for the Frobenius case. In this sense the problem stays similar.

A great advantage of the Haff loss is that it is one of the few for which an unbiased estimator of the risk is known, at least in the orthogonally invariant case. There is a rich body of literature behind that construction \citep{Haff77,Haff79,Haff80}, in different shapes and under different conditions. A remarkable feature is that if we collect and split the terms of the URE between the terms of leading and smaller order, the dominant part does not depend on the derivatives of the eigenvalue estimators. More precisely, we have this construction.

\begin{thm}\label{thm:URE} Let $n\geq p+1$. Then for any weak spiked estimator $\hat\Sigma\in V_p(\hat\Gamma)$ whose spiked rank $\hat\rho$ is independent of $S$, we find its Haff risk to satisfy
$\text{E}\big[L_H(\hat\Sigma,\Sigma)\big]=\text{E}\big[F+G\big]$ with $\text{E}\big[|F+G|\big]<\infty$,
where 
{\setlength{\abovedisplayskip}{5pt}
\setlength{\belowdisplayskip}{10pt}
\setlength{\mathindent}{10pt}\begin{align*}&
F=F\Big(l,\hat\rho,\hat\gamma,\hat\sigma^2\Big)
\;\;\text{ and }\;\;
G=G\Big(l,\hat\rho,\hat\gamma,\hat\sigma^2,\frac{\partial\hat\gamma}{\partial l},\frac{\partial\hat\sigma^2}{\partial l},\frac{\partial^2\hat\gamma}{\partial l^2},\frac{\partial^2\hat\sigma^2}{\partial l^2}\Big)
\end{align*}}%
are functionals that do not depend on $\Sigma$. In addition, if the estimator is strong in the sense that $\hat\Sigma\in\bar V_p(\hat\Gamma)$, then asymptotically $F$ is the dominant term and $G$ the dominated term, respectively, in the sense that
\begin{align*}&
\lim_{n\rightarrow\infty}\;\E{\vphantom{\Big\vert}\!\left|F\right|}<\infty
\;\text{ and }\;
\lim_{n\rightarrow\infty}\;p\E{\vphantom{\Big\vert}\!\left|G\right|}<\infty.
\end{align*}
Explicit expressions for $F$ and $G$ are given in (\ref{eq:CF-ureF})--(\ref{eq:CF-ureG}).
\end{thm}
In practice, we found the decomposition most useful for a deterministic spiked rank $\hat\rho=r$, in which case we might consider estimates of the form $\hat\Gamma_r+\sigma_r^2I$; this is the approach we take in Section \ref{sec:A} when constructing a specific estimator. But it is reasonable to think of a context in which some estimate of the true rank $\rho$ based on prior independent data is available, in which case the construction applies equally.

Now fix a weak spiked eigenvalue estimator $\hat\Gamma$, and consider the task of finding a $\hat\sigma^2$ that minimizes the Haff risk: minimizing the  construction from Theorem \ref{thm:URE} makes the task plausible. Since $G$ depends on the derivatives of $\hat\sigma^2$, formally proceeding with calculus of variations would yield a minimizer that depends on the unknown density of the eigenvalues which is, of course, unavailable. However since the dominant part only depends on $\hat\sigma^2$ itself, one can obtain a minimizer of $\E{F}$ whose expression is independent of $\Sigma$.

\begin{prop}\label{prop:MINM} Let $n\geq p$. If $\hat\Sigma\mapsto\E{F}$ has a minimum over $V_p(\hat\Gamma)$, it is given by $\tilde\Sigma=\hat\Gamma+\tilde\sigma^2I$, where $\tilde\sigma^2=A/B$ where
{\setlength{\mathindent}{5pt}\begin{align*}
&A=\;\frac{n-p-1}{np}\sum_{c=1}^{p}\frac{1}{l_c}
-\frac{(n-p-1)(n-p-2)}{n^2p}\sum_{k=1}^{\hat\rho}\frac{\hat\gamma_k}{l_k^2}
\\&\qquad
+\frac{n-p-1}{n^2p}\sum_{k=1}^{\hat\rho}\sum_{c=1}^p\frac{\hat\gamma_k}{l_k}\frac{1}{l_c}
-2\frac{n-p-1}{n^2p}\sum_{k=1}^{\hat\rho}\sum_{c={\hat\rho}+1}^p\frac{1}{l_c}\frac{\hat\gamma_k}{l_k-l_c}
\\&\qquad
+\frac3{n^2p}\sum_{k\neq b}^{\hat\rho}\sum_{c={\hat\rho}+1}^p\frac{1}{l_k-l_c}\frac{\hat\gamma_b}{l_b-l_c}
-\frac3{n^2p}\sum_{k=1}^{\hat\rho}\sum_{c\neq d={\hat\rho}+1}^p\frac{1}{l_k-l_c}\frac{\hat\gamma_k}{l_k-l_d}
\\&\qquad
-\frac3{n^2p}\sum_{k\neq b}^{\hat\rho}\sum_{c={\hat\rho}+1}^p\frac{\hat\gamma_k-\hat\gamma_b}{l_k-l_b}\frac{1}{l_k-l_c},
\\
&B=\;
\frac{(n-p-1)(n-p-2)}{n^2p}\sum_{c=1}^{p}\frac{1}{l_c^2}
-\frac{n-p-1}{n^2p}\sum_{c=1}^p\frac{1}{l_c}\sum_{c=1}^p\frac{1}{l_c}.
\end{align*}}%
In addition, if $n\geq 2p+2$, a minimum must exist (and therefore at $\tilde\Sigma$).
\end{prop}

The matter of whether a minimizer should exist at all in a given context is delicate. A proof of existence for some large class of covariance matrices would be quite interesting. In the spiked case, the remarks following Lemma \ref{lem:LLN} hint at a plausible approach.

\section{Properties}\label{sec:P}

The previous chapter was concerned with the construction of a good estimator $\tilde\sigma^2$ that satisfies some optimality property, namely minimizing the dominant part of the Haff URE over $V_p(\hat\Gamma)$. Let us now turn our attention to its performance in estimating $\sigma^2$ under spikedness. We will make repeated use of the following lemma, which extends the results of \cite{Nadler08}. 

\begin{lem}\label{lem:LLN} Suppose the underlying sequence of covariance matrices $\{\Sigma_p\}$ is spiked and $p_n/n\rightarrow c\in(0,1)$.
\begin{enumerate}[(i)]
\item If $\gamma_\rho/\sigma^2>\sqrt{c}$, then for any $1\leq k\leq \rho$,
\[
\frac1{p-\rho}\sum_{c=\rho+1}^p\frac{l_c}{l_k-l_c}\xrightarrow[n\rightarrow\infty]{\text{a.s.}}\frac{\sigma^2}{\gamma_k};
\]
\item For any $m>1$,
\[\frac1{p-\rho}\sum_{c=\rho+1}^p\frac{1}{l_c^m}\xrightarrow[n\rightarrow\infty]{\text{a.s.}}\frac1{(1-c)^{2m-1}}\frac1{\sigma^{2m}}.\]
\end{enumerate}
\end{lem}
The supercriticality assumption $\gamma_\rho/\sigma^2>\sqrt{c}$ in (i) is necessary for the expression to converge. Two remarks are in order. First, as a consequence of this result, it is easy to show that
\begin{align*}&
\frac1{p-\rho}\sum_{c=\rho+1}^p\frac{1}{l_k-l_c}\xrightarrow[n\rightarrow\infty]{\text{a.s.}}\frac1{\gamma_k+c\sigma^2},
\\&
\frac1{p-\rho}\sum_{c=\rho+1}^p\frac{1}{l_c(l_k-l_c)}\xrightarrow[n\rightarrow\infty]{\text{a.s.}}\frac1{1-c}\frac{\gamma_k}{\sigma^2(\gamma_k+c\sigma^2)^2},
\end{align*}
a result we will use later in Section \ref{sec:A}. Second, in connection with the proof of Proposition \ref{prop:MINM}, we see that when the underlying sequence of covariance matrices $\{\Sigma_p\}$ is spiked
\begin{align*}
\frac1{n^2}\left[(n-p-2)\sum_{c=1}^{p}\frac{1}{l_c^2}
-\left(\sum_{c=1}^p\frac{1}{l_c}\right)^2\right] \xrightarrow[n\rightarrow\infty]{\text{a.s.}}\frac{c}{(1-c)\sigma^4}>0,
\end{align*}
by Lemma \ref{lem:LLN}. From (\ref{eq:MINM-secder}), one would therefore expect the estimator also to be a minimizer under spikedness. Although we haven't been successful in formalizing this intuition, this could be a plausible approach towards proving existence of minimizers for spiked covariance matrices.

Let us now turn our attention to the behavior of $\tilde\sigma^2$. The following theorem summarizes important aspects of its asymptotic behavior.

\begin{thm}\label{thm:NORM} Suppose the underlying sequence of covariance matrices $\{\Sigma_p\}$ is spiked and $p_n/n\rightarrow c\in(0,1)$ with $\gamma_\rho/\sigma^2>\sqrt{c}$. For a given weak $\hat\Gamma$, let $\tilde\sigma^2$ be the associated minimizer of Proposition \ref{prop:MINM}. Then
\begin{enumerate}[(i)]
\item\label{eq:NORM-1} If $\hat\rho$ a.s. converges to a finite constant, then
$\tilde\sigma^2\xrightarrow[n\rightarrow\infty]{\text{a.s.}}\sigma^2$.
\item\label{eq:NORM-2} If $\hat\rho$ is strongly consistent and for all $1\leq k\leq\rho$, $\hat\gamma_k$ a.s. converges to some finite constant, then we have bounds $X^-_n\leq n(\tilde\sigma^2-\sigma^2)\leq X_n^+$ with
{\setlength{\mathindent}{5pt}\begin{align*}&
X^-_n\xrightarrow[n\rightarrow\infty]{\mathcal{D}}\text{N}\left(\mu^-,\frac{2c(1+c)^2}{(1-c)^4}\sigma^4\right)\!,
\quad
X^+_n\xrightarrow[n\rightarrow\infty]{\mathcal{D}}\text{N}\left(\mu^+,\frac{2c(1+c)^2}{(1-c)^4}\sigma^4\right)\!,
\end{align*}}%
where $\mu^-$ and $\mu^+$ have explicit expressions given in (\ref{eq:NORM-mu+})--(\ref{eq:NORM-mu-}).
\end{enumerate}
\end{thm}

An immediate consequence of this result is that $\tilde\sigma^2$ estimates $\sigma^2$ with rate $n$ in our absolute error loss $L(\hat\sigma^2,\sigma^2)=|\hat\sigma^2-\sigma^2|$. We should mention that, given good estimators of $\rho$ and $\gamma_k$, one could perhaps build approximate high-dimensional confidence intervals for $\sigma^2$ with part (\ref{eq:NORM-2}) of Theorem \ref{thm:NORM}. We will not investigate this further, but rather turn our attention to minimax rates for the noise estimation problem. For any spiked sequence $\{\Sigma_p\}$, define a $\delta$-ball of order $r$ as 
{\setlength{\mathindent}{10pt}\begin{align*}
\text{B}_r(\Sigma,\delta)=
\left\{\{\Sigma_p'\} \text{ spiked }\;\bigg\vert\; \left|\lambda_i(\Sigma_p)-\lambda_i(\Sigma'_p)\right|< \delta\frac{\lambda_i(\Sigma_p)}{n^r}\quad\forall p>0\right\}\!.
\end{align*}}%
We start with a lemma. Recall that $\text{d}_\text{TV}$ stands for the total variation distance between two probability measures.

\begin{lem}\label{lem:TV}
Let $\{\Sigma_p\}$ be a spiked sequence of covariance matrices and $M>0$. Then, as $p_n/n\rightarrow c\in(0,1)$,
\[
\lim_{n\rightarrow\infty}\sup_{\Sigma'\in\text{B}_r(\Sigma,2M)}\!\text{d}_\text{TV}\!\left(\vphantom{\bigg\vert}\text{N}(0,\Sigma_p)^n,\text{N}(0,\Sigma'_p)^n\right)\leq\; \sqrt{1-\exp\left(-\frac{cM^2}{2}\right)}
\]
when $r=1$, while this limit is zero when $r>1$.
\end{lem}

Using this result, we now proceed to show a lower bound on the local minimax rate of convergence for estimating $\sigma^2$, using the classic two-point test argument of \cite{LeCam73}.

\begin{thm}\label{thm:MRLB} Say the underlying sequence of covariance matrices $\{\Sigma_p\}$ is spiked and $p_n/n\rightarrow c\in(0,1)$. Let $\epsilon>0$ and define $$M_\epsilon=\sqrt{-\frac2c\log\Big(1-(1-4\epsilon)^2\Big)}.$$  Then no estimator can estimate $\sigma^2$ with speed $\sigma^2M_\epsilon/n$ over the shrinking neighborhoods B$(\Sigma_p,2M_\epsilon)$, in the sense that
\begin{align*}
\underset{n\rightarrow\infty}{\lim\inf}\inf_{\hat\sigma^2}\sup_{\substack{\Sigma'\in\text{B}_1(\Sigma,2M_\epsilon)}}
\text{P}_{\Sigma'_p}\left[\vphantom{\Bigg\vert}|\hat\sigma^2-\sigma^{2\prime}|> \sigma^2\frac{M_\epsilon}n\right]\geq\epsilon.
\end{align*}
\end{thm}

Thus, the minimax rate of estimation of $\sigma^2$ over $n$-shrinking neighborhoods cannot be faster than $O_P(1/n)$ (so in particular over, say, fixed neighborhoods.) Using Theorem \ref{thm:NORM}, we can show our noise estimator $\tilde\sigma^2$ essentially achieves this rate, in the sense that it is $o_P(1/n^r)$  over $n^r$-shrinking neighborhoods for any $r>1$.

\begin{prop}\label{prop:RATE}
Let the underlying sequence of covariance matrices $\{\Sigma_p\}$ be spiked and $p_n/n\rightarrow c\in(0,1)$ with $\gamma_\rho/\sigma^2>\sqrt{c}$. For a given weak $\hat\Gamma$, let $\tilde\sigma^2$ be the associated extremizer of proposition \ref{prop:MINM}. If $\hat\rho$ is strongly consistent and for all $1\leq k\leq\rho$, $\hat\gamma_k$ a.s. converges to some finite constant, then for any $r>1$ and $M>0$, $\tilde\sigma^2$ estimates $\sigma^2$ with rate at least $\sigma^2M/n^r$ over the shrinking neighborhoods $\text{B}_r(\Sigma,2M)$, in the sense that
\begin{align*}
\lim_{n\rightarrow\infty}\sup_{\substack{\Sigma'\in\text{B}_r(\Sigma,2M)}}
\text{P}_{\Sigma'_p}\left[\vphantom{\Bigg\vert}|\tilde\sigma^2-\sigma^{2\prime}|> \sigma^2\frac{M}{n^r}\right]=0.
\end{align*}
\end{prop}

Thus we can conclude that, despite choosing our noise estimator to minimize a covariance problem, good behavior has been transferred to the noise estimation problem, which is not surprising in light of (\ref{eq:I-dual}). In particular, we see that by Theorem \ref{thm:NORM} (\ref{eq:NORM-1}), strong consistency of the noise estimator follows even when the eigenvalues of $\hat\Gamma$ itself are not consistent -- a robustness which is certainly welcome.

\section{Application}\label{sec:A}

Having built and analyzed our noise estimator, we now proceed to illustrate our construction by building a specific covariance estimator. We hope this concrete example will help clarify the approach taken and its behavior in the covariance problem.

\subsection{Example} We build a spiked covariance estimator as follows. The first step is to specify an asymptoticaly negligible spiked component $\hat\Gamma$. For $r$ be some fixed rank strictly smaller than $p$, we take $\tilde\Gamma_r=O\diag{\tilde\gamma}O'$ with
\begin{align*}&
\tilde\gamma_{k}=
\sum_{c=r+1}^p\!l_c
\left(\sum\limits_{c=r+1}^p\frac{l_c}{l_k-l_c}\right)^{-1}
\end{align*}
for $1\!\leq\!k\!\leq\!r$, and $0$ otherwise. These estimators are strongly consistent when $r=\rho$, as we will soon show using Lemma \ref{lem:LLN}; this is the main motivation for our choice. Note that this choice does not quite fit within the framework considered by \cite{Donoho14}, since it is not a function of $l_k$ only. With this choice of spiked part, let $\tilde\sigma^2_r$ be the minimizer of Proposition \ref{prop:MINM} associated with our spiked component. We then have a family $r\rightarrow \tilde\Sigma_r=\tilde\Gamma_r+\tilde\sigma^2_rI$ for all $0\!\leq\!r\!<\!p$, which we naturally extend to the $r\!=\!p$ case through $\tilde\Sigma_p=\tilde\Gamma_p=S$. 

Next, we select the rank $r$ based on the data. Motivated again by the results of Lemma \ref{lem:LLN}, we define the rank estimator
{\setlength{\mathindent}{5pt}\begin{align*}
\tilde\rho=\underset{0\leq r\leq p}{\arg\min}\left\{\frac{\1{r<p}}{l_{r+1}}\frac{(1+\sqrt{p/n})^2}{p-r}\hspace{-5pt}\sum_{c=r+1}^p l_c\geq 1, 
\quad
\big|F_r+G_r\big|\leq \frac {p+1}n\right\},
\end{align*}}%
with $F_r$, $G_r$ the $F$, $G$ of Theorem \ref{thm:URE} applied to $\tilde\Gamma_r$ and $\tilde\sigma^2_r$. This choice aims to select the smallest rank that lies both above the critical threshold and yields improvement in Haff risk over $S$. Since $r=p$ satisfies both criteria, the set is never empty and in the worst case we simply do not correct the eigenvalues of $S$. This will happen when there is strong departure from spikedness, which means that the construction is in some sense robust to this situation: it exploits it when present and reverts to $S$ when not.

Finally, we simply set $\tilde\Sigma=\tilde\Sigma_{\tilde\rho}$ as our estimator. In practice, the computation is straightforward, since everything is in closed-form, with polynomial complexity. An implementation is available as an R package at \url{http://stat.cornell.edu/~chetelat}. At the same time, as previously hinted, the estimator has strongly consistent eigenvalues under spikedness. The proof is a simple application of results from Section \ref{sec:P}.
\begin{prop}\label{prop:CONS}
If the underlying sequence of covariance matrices $\{\Sigma_p\}$ is spiked and $p_n/n\rightarrow c\in(0,1)$ with $\gamma_\rho/\sigma^2>\sqrt{c}$, then $\tilde\rho$, $\tilde\sigma^2$ and $\tilde\gamma_k$ for $1\leq k\leq \rho$ are all strongly consistent.
\end{prop}

This stands, of course, in contrast with the eigenvalues of the sample covariance matrix $S$, which converge to the wrong values (\ref{eq:l-limit}) in spiked settings.

\subsection{Numerical Comparisons}

We display the performance of the constructed estimator through simulations. The setting is as follows. We fix the dimension sample ratio $c$ at 0.5 and vary $n$, $p$. For each $n$, $p$, we simulate data from a normal $\text{N}(0,\Sigma)$ and approximate its Haff and Frobenius risk using a law of large numbers approximation with 100 iterations. Four true covariance matrices $\Sigma$ are considered. The first is a spiked setting $\Sigma=\diag{5,4,3,2,1,...,1}$, while the other three correspond to autoregressive settings $\Sigma_{ij}=\kappa^{|i-j|}$ for $\kappa=0.05$, $0.5$ and $0.95$. The case $\kappa=0.95$ is particularly difficult for the constructed estimator as it is very far from spikedness.

The risks are computed for $S$, our estimator $\tilde\Sigma$ and two benchmark competitors. The first is Stein's isotonized covariance estimator, with well regarded overall performance. We follow the implementation of \cite{Lin85}. The second is the popular linear shrinkage covariance estimator of  \cite{LedoitWolf04}, specifically designed for high-dimensional settings. We plot the risks and the gain in risk with respect to $S$, defined as $\text{Risk}(S)/\text{Risk}(\hat\Sigma)-1$. The computations were performed using the R package, and the results are given as Figures \ref{plot:spiked}-\ref{plot:ar095}. Blue corresponds to $S$, red to $\tilde\Sigma$, green to Stein's isotonized estimator and yellow to the Ledoit-Wolf estimator. 

\begin{figure}[t]
\caption{Spiked covariance setting.}
\label{plot:spiked}
\includegraphics[trim = 2cm 12.5cm 2cm 1cm, clip, scale=0.7]{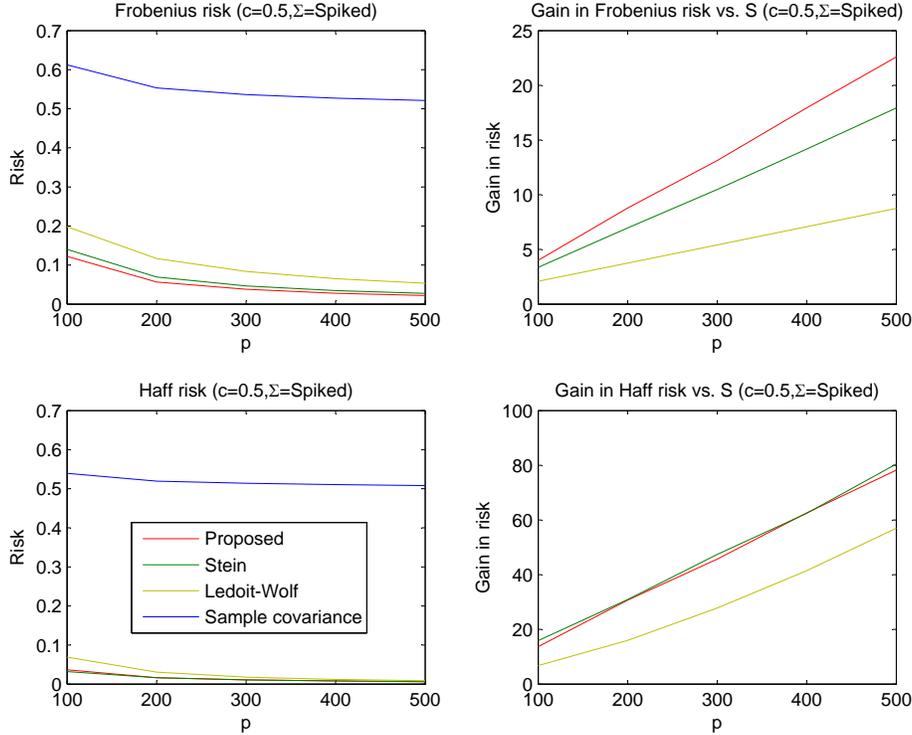}
\end{figure}

\begin{figure}[t]
\caption{AR(0.05) setting. The sample covariance matrix is omitted.}
\label{plot:ar005}
\includegraphics[trim = 1.5cm 12.5cm 2cm 1cm, clip, scale=0.7]{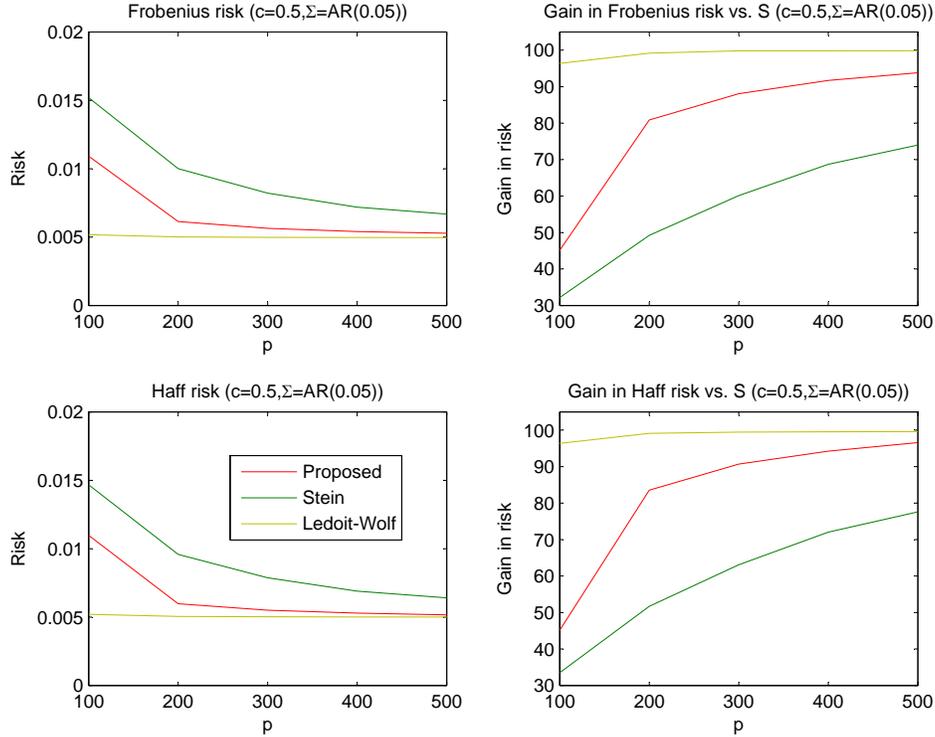}
\end{figure}

\begin{figure}[t]
\caption{AR(0.50) setting.}
\label{plot:ar050}
\includegraphics[trim = 2cm 12.5cm 2cm 1cm, clip, scale=0.7]{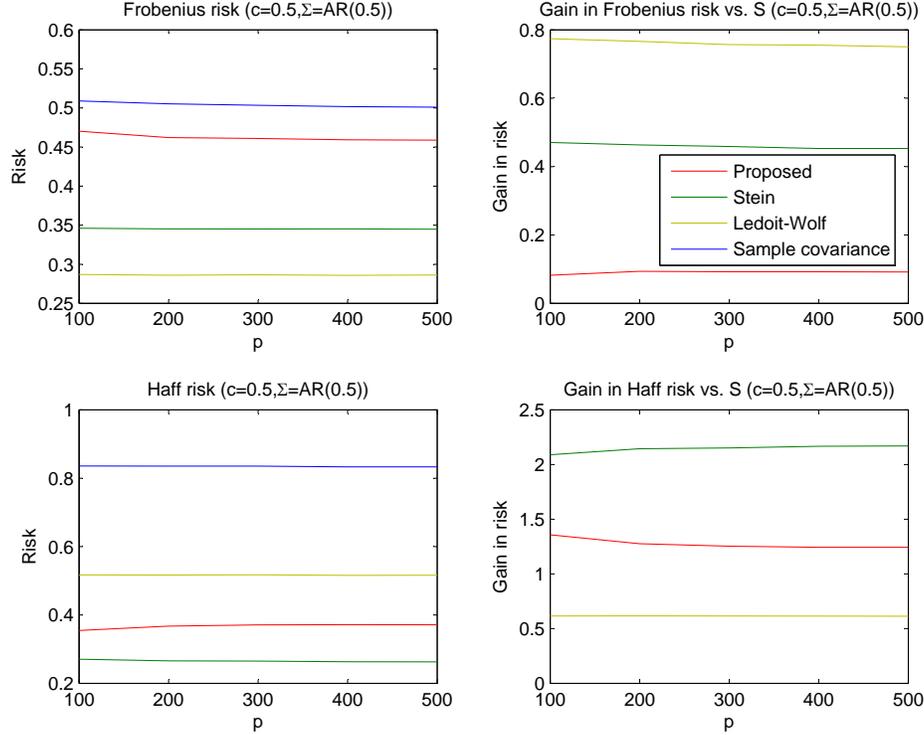}
\end{figure}

\begin{figure}[t]
\caption{AR(0.95) setting.}
\label{plot:ar095}
\includegraphics[trim = 2cm 12.5cm 2cm 1cm, clip, scale=0.7]{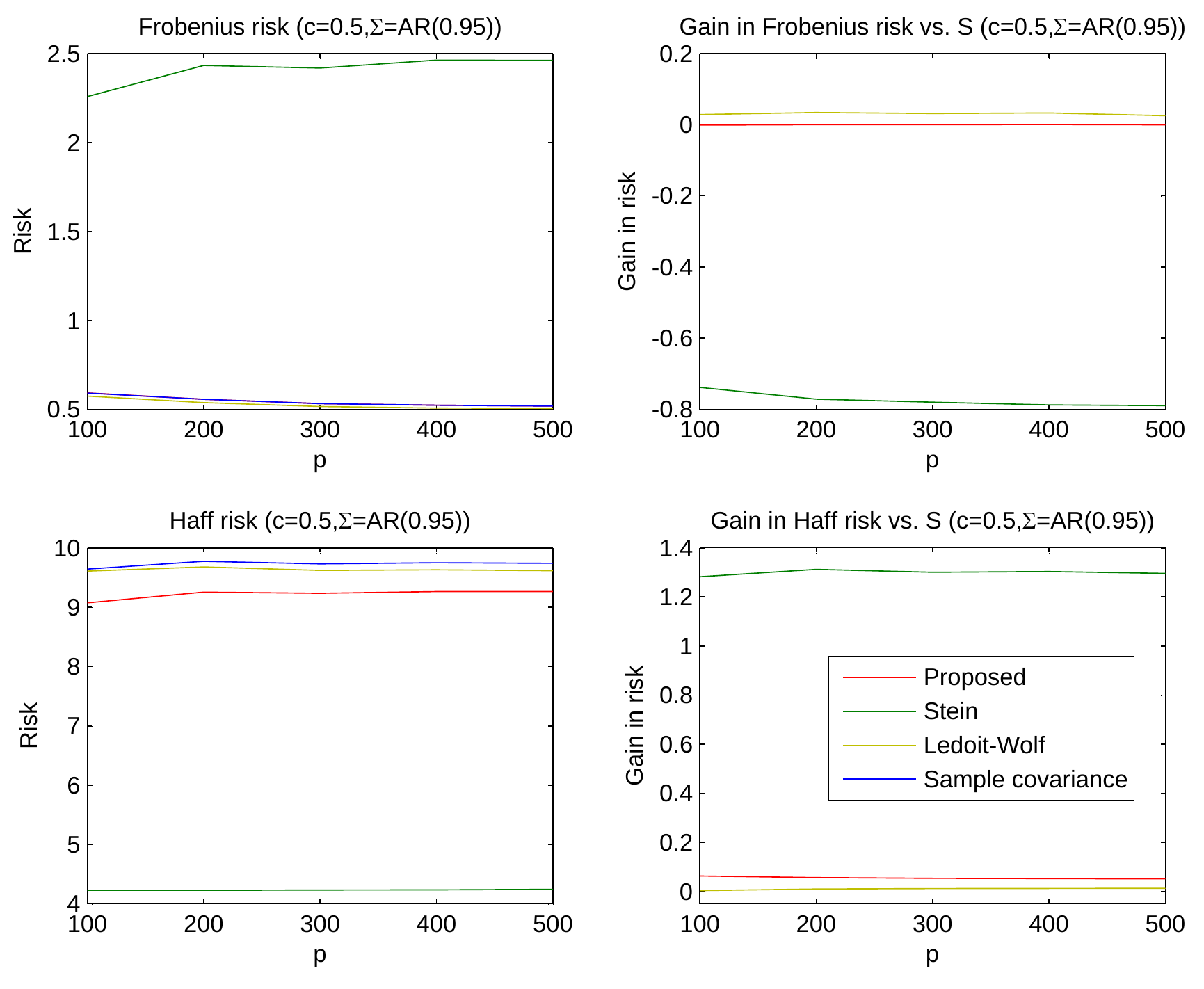}
\end{figure}

We can see from the results of the simulations that the expected good performance in Haff loss of our estimator seems to translate well into the more standard Frobenius loss. The estimator performs particularly well in supercritical spiked settings, with a $23$-times improvement over $S$ in Frobenius risk for $p=500$, $n=1000$ in our setting. In addition, the estimator is quite robust to deviations from spikedness, as even in worst-case scenarios such as an AR(0.95) setting, we do not do worse than $S$ in Haff or Frobenius risk. There is therefore little to lose by using it rather than the sample covariance matrix, and as far as such a thing can exist, it could be advocated as some kind of generic high-dimensional covariance estimator.

\section{Comments}\label{sec:D}

In this work we considered two parallel high-dimensional problems, the estimation of noise in principal components analysis under absolute error loss and the estimation of a spiked covariance matrix under Frobenius loss. We proposed a variational solution, by restricting ourselves to regular estimators and minimizing an unbiased covariance risk estimator in the invariant analogue of the loss. The resulting noise estimator was shown to be strongly consistent and almost asymptotically normal and minimax for the noise problem, and we used the construction to build a robust spiked covariance estimator with good simulation performance. Beyond this, however, there are several aspects of our solution that warrant further discussion. 

First, we assumed throughout this work that the underlying data was normal, and the construction and proofs depend quite heavily on it. This could be a point of discord between practice and the theory outlined here. However, we feel that, unlike many statistical problems where normality is convenient but unrealistic, it is quite natural here. Indeed, the construction and its properties only depends on the data through the eigenstructure of $S$, unlike other estimators such as the one of \cite{LedoitWolf04}. The sample covariance matrix being an empirical average, one can expect it to behave asymptotically like a Wishart, and in that sense the assumption does not appear particularly restrictive.

Another assumption running through the work is that although we are in high-dimensions, we keep $p\leq n$. The extension to a $p>n$ setting is attractive, as in addition to the properties described above, a corresponding robust spiked covariance estimator would be automatically invertible, in contrast with $S$. The single obstacle appears to be the absence of an appropriate unbiased risk estimator for the Haff loss when $p>n$. This is therefore more an obstruction by knowledge than mathematics, as if such a construction would be found, the method outlined in this work could easily be applied.

As we considered minimization of a covariance loss, it might be surprising that we did not present any results on the behavior of the estimator in the covariance problem. We strongly believe that the Haff risk must tend to zero under spikedness since, as some algebra shows, the unbiased risk estimator of our estimate tends a.s. to zero. This is quite interesting since the Haff risk of $S$ equals, in contrast, $(p+1)/n\rightarrow c>0$. Unfortunately, we haven't been able to prove this statement. Although the literature on the probabilistic behavior of Wishart eigenvalues is extensive, it is more scant on their $L^1$ behavior, and this limits what can be proven as of now.

In Section \ref{sec:C}, we considered an invariant analogue of the Frobenius loss $\|\hat\Sigma-\Sigma\|_F^2/p$, the Haff loss $\|\hat\Sigma\Sigma^{-1}-I\|_F^2/p$ which allowed for the existence of an unbiased risk estimator. Since our estimator is particularly adapted to this invariant covariance loss, it might also be of interest to study an invariant noise loss such as $|\hat\sigma^2/\sigma^2-1|$.

We did not tackle the problem of selecting the spiked eigenvalue estimators $\hat\gamma_k$ in an optimal way, beyond the suggestion in Section \ref{sec:A}. The recent work of \cite{Donoho14} could offer a solution. The authors consider the spiked covariance estimation problem where the noise is known and fixed at $\sigma^2=1$, and look for optimal shrinking of the spiked eigenvalues $l_k$, $1\leq k\leq \rho$. In the Frobenius and Haff losses, their optimal estimators coincide and equal
\begin{align*}
\hat\gamma_k=\left[l_k-1+\frac{cl_k}{l_k-1}\right]\frac{1-c/(l_k-1)^2}{1+c/(l_k-1)}
\end{align*}
for $l_k>(1+\sqrt{c})^2$. An appealing feature of this estimate is that it accounts for the deterministic angles between the top sample and population eigenvectors. It would be interesting to study the behavior of the noise estimator $\tilde\sigma^2$ from Theorem \ref{prop:MINM} applied to these spiked estimators, with perhaps adjustments for not knowing $\sigma^2$.

Finally, we should remark that our construction automatically provides well-conditioned covariance estimators, which is quite important for applications. Therefore, when the parameter of interest is the precision rather than the covariance matrix, using $\tilde\Sigma^{-1}$ as estimator appears reasonable, although we currently do not have any formal results on its behavior for this problem.
\section{Technical Results and Proofs}\label{sec:L}

\subsection{Proofs for Section \ref{sec:C}}

The following Stein-Haff identity is used to compute an unbiased estimator of risk for orthogonally invariant estimators in proposition \ref{thm:URE}. The general identity dates back to  \cite{Haff79} and its specialization to orthogonally invariant estimators for $n\geq p$ first implicitly used by \cite{Sheena95}. Unfortunately, the approach taken by the author requires regularity conditions that are difficult to verify in practice (conditions 1--3 in his Section 1 and 2). The following lemma follows the approach of \cite{Konno09} and \cite{KubokawaSrivastava08} to obtain the same $n\geq p$ result, but under weaker, simpler conditions.

\begin{lem}\label{lem:KONN} Let $W\sim W_p(n,\Sigma)$ with $n\geq p$, and let $W/n=OLO'$ be the spectral decomposition of the associated sample covariance matrix. Let $\psi_1(L),...,\psi_p(L)$ be differentiable functions of the eigenvalues of $W/n$ satisfying:
\begin{align}
&\E{\left|\sum_{k=1}^p \frac{n-p-1}n\frac{\psi_k}{l_k}+\frac{2}{n}\sum_{k=1}^p\frac{\partial \psi_k}{\partial l_k}+\frac1{n}\sum_{k\neq b}^p\frac{\psi_k-\psi_b}{l_k-l_b}\right|}<\infty.
\label{eq:KONN-reg}
\end{align}
Define $\Psi=\text{diag}(\psi_1,...,\psi_p)$. Then
{\setlength{\mathindent}{10pt}\begin{align*}
\E{\tr{\Sigma^{-1}O\Psi O'}}
=
\E{\sum_{k=1}^p \frac{n-p-1}n\frac{\psi_k}{l_k}+\frac{2}{n}\sum_{k=1}^p\frac{\partial \psi_k}{\partial l_k}+\frac1{n}\sum_{k\neq b}^p\frac{\psi_k-\psi_b}{l_k-l_b}}\!.
\end{align*}}
\end{lem}

\begin{proof}[\bf Proof]
We use Lemma 3 in \cite{ChetelatWells12}. Decompose $W=X'X$ for some $X\sim N_{n\times p}(0,I_n\otimes \Sigma)$. In the spirit of Lemma 4.1 in \cite{Konno09}, we find
\begin{align*}
&(dX')X+X'dX=d(X'X)=n(dO)LO'+nOdLO'+nOLdO'
\\&\quad\Rightarrow\qquad O'\left[(dX')X+X'dX\right]O=nO'dOL+nL(dO')O+ndL.
\end{align*}
Since $O'dO+(dO')O=0$, we get
\begin{align*}
&O'\left[(dX')X+X'dX\right]O=nO'dOL-nLO'dO+ndL,
\end{align*}
so that for $k\neq l$,
\begin{align*}
&O'dO_{kl}=\frac1n\frac1{l_l-l_k}O'\left[(dX')X+X'dX\right]O_{kl},
\\&
dL_{kk}=\frac1nO'\left[(dX')X+X'dX\right]O_{kk}
\end{align*}
and $O'dO_{kk}=0$. Then 
\begin{align}
&\frac{\partial l_k}{\partial X_{ij}}=
\frac1n\sum_{\alpha,\beta,\gamma}O'_{k\alpha}\frac{X'_{\alpha\beta}}{dX_{ij}}X_{\beta\gamma}O_{\gamma k}
+\frac1n\sum_{\alpha,\beta,\gamma}O'_{k\alpha}X'_{\alpha\beta}\frac{X_{\beta\gamma}}{dX_{ij}}O_{\gamma k}
\notag\\&\qquad=
\frac2n\sum_{\gamma}O'_{kj}X_{i\gamma}O_{\gamma k}
\label{eq:KONN-partial1}
\end{align}
and
\begin{align}
&\frac{\partial O_{kl}}{\partial X_{ij}}
=
\frac1n\sum_{\alpha\neq l,\beta,\gamma,\epsilon} O_{k\alpha}\frac1{l_l-l_\alpha}O'_{\alpha\beta}\left[\frac{\partial X'_{\beta\gamma}}{\partial X_{ij}}X_{\gamma\epsilon}+X'_{\beta\gamma}\frac{\partial X'_{\gamma\epsilon}}{\partial X_{ij}}\right]O_{\epsilon l}
\notag\\&\qquad=
\frac1n\sum_{\alpha\neq l,\beta} O_{k\alpha}\frac{O'_{\alpha j}O_{\beta l}+O'_{\alpha\beta}O_{jl}}{l_l-l_\alpha}X_{i\beta}.
\label{eq:KONN-partial2}
\end{align}

Now define $\tilde X=X\Sigma^{-1/2}$ and $H=\frac1n\Sigma^{1/2} OL^{-1}\Psi O'\Sigma^{-1/2}$ -- we need to compute $\text{div}_{\text{vec}(\tilde X)}\;\text{vec}\left(\tilde X H\right)$. We find
{\setlength{\mathindent}{5pt}\begin{align}
&\text{div}_{\text{vec}(\tilde X)}\;\text{vec}\left(\tilde X H\right)
=
\sum_{\alpha,i,j}\frac{\partial}{\partial \tilde X_{\alpha i}} \left\lbrace \tilde X_{\alpha j} H_{ji} \right\rbrace 
=
n\sum_{i}H_{ii}+\sum_{\alpha,j}\tilde X_{\alpha j} \frac{\partial H_{ji}}{\partial \tilde X_{\alpha i}} 
\notag\\&\quad=
\sum_{\gamma}\frac{\psi_\gamma}{l_\gamma}+\frac1n\sum_{\alpha,\beta,i,j,k,l}\tilde X_{\alpha j} \Sigma^{1/2}_{\beta i}\Sigma^{1/2}_{jk} \frac{\partial}{\partial X_{\alpha \beta}} \left\lbrace  OL^{-1}\Psi O'_{kl} \right\rbrace\Sigma^{-1/2}_{li}
\notag\\&\quad=
\sum_{\gamma}\frac{\psi_\gamma}{l_\gamma}
+\frac1n\sum_{\alpha,k,l} X_{\alpha  k}  \frac{\partial}{\partial X_{\alpha l}} \left\lbrace  OL^{-1}\Psi O'_{kl} \right\rbrace
\label{eq:KONN-ure-xpr}
\\&\quad=
\sum_{\gamma}\frac{\psi_\gamma}{l_\gamma}
+\frac1n\sum_{\alpha,k,l,\beta} X_{\alpha  k}  \frac{\partial O_{k\beta}}{\partial X_{\alpha l}}\left[L^{-1}\Psi\right]_{\beta\beta}O'_{\beta l}
\notag\\&\qquad
+\frac1n\sum_{\alpha,k,l,\beta} X_{\alpha  k}  O_{k\beta}\frac{\partial \left[L^{-1}\Psi\right]_{\beta\beta}}{\partial X_{\alpha l}}O'_{\beta l}
+\frac1n\sum_{\alpha,k,l,\beta} X_{\alpha  k}  O_{k \beta}\left[L^{-1}\Psi\right]_{\beta\beta}\frac{\partial O'_{\beta l}}{\partial X_{\alpha l}}.\notag
\end{align}}
Using (\ref{eq:KONN-partial1}) and (\ref{eq:KONN-partial2}), we obtain
{\setlength{\mathindent}{5pt}\begin{align*}
&\quad=
\sum_{\gamma}\frac{\psi_\gamma}{l_\gamma}
+\frac1{n^2}\sum_{\alpha,k,l,\beta,\gamma\neq \beta,\epsilon}X_{\alpha  k}  
O_{k\gamma}\frac{O'_{\gamma l}O_{\epsilon\beta}+O'_{\gamma\epsilon}O_{l\beta}}{l_\beta-l_\gamma}X_{\alpha\epsilon}
\left[L^{-1}\Psi\right]_{\beta\beta}O'_{\beta l}
\notag\\&\quad\quad
+\frac2{n^2}\sum_{\alpha,k,l,\beta,\gamma,\epsilon}X_{\alpha  k}  O_{k\beta}
O'_{\gamma l}X_{\alpha\epsilon}O_{\epsilon\gamma}
\frac{\partial \left[\psi_\beta/l_\beta\right]}{\partial l_\gamma}O'_{\beta l}
\notag\\&\quad\quad
+\frac1{n^2}\sum_{\alpha,k,l,\beta,\gamma\neq \beta,\epsilon} X_{\alpha  k}  O_{k \beta}\left[L^{-1}\Psi\right]_{\beta\beta}
O_{l\gamma}\frac{O'_{\gamma l}O_{\epsilon\beta}+O'_{\gamma\epsilon}O_{l\beta}}{l_\beta-l_\gamma}X_{\alpha\epsilon}
\notag\\&\quad=
\sum_{\gamma}\frac{\psi_\gamma}{l_\gamma}
+\frac1{n}\sum_{\gamma\neq \beta}
 \frac{l_\gamma\psi_\beta}{(l_\beta-l_\gamma)l_\beta}
+\frac2{n}\sum_{\gamma}l_\gamma
\frac{\partial \left[\psi_\gamma/l_\gamma\right]}{\partial l_\gamma}
+\frac1{n}\sum_{\gamma\neq \beta} 
\frac{\psi_\beta}{l_\beta-l_\gamma}
\notag\\&\quad=
\frac{n-p-1}{n}\sum_{\gamma}\frac{\psi_\gamma}{l_\gamma}
+\frac2n\sum_{\gamma}\frac{\partial \psi_\gamma}{\partial l_\gamma}
+\frac1{n}\sum_{\gamma\neq \beta} 
\frac{\psi_\beta-\psi_\gamma}{l_\beta-l_\gamma}.
\end{align*}}
By (\ref{eq:KONN-reg}), we conclude
\begin{align*}
&\E{\left|\text{div}_{\text{vec}(\tilde X)}\;\text{vec}\left(\tilde X H\right)\right|}
\\&\qquad= 
\E{\left|\sum_{k=1}^p \frac{n-p-1}n\frac{\psi_k}{l_k}+\frac{2}{n}\sum_{k=1}^p\frac{\partial \psi_k}{\partial l_k}+\frac1{n}\sum_{k\neq b}^p\frac{\psi_k-\psi_b}{l_k-l_b}\right|}
<\infty.
\end{align*}
Therefore, we can apply Lemma 3 in \cite{ChetelatWells12} to $G=\frac1n OL^{-1}\Psi O'$, which holds for any $(p, n)$. We obtain that
\begin{align*}
&\E{\tr{\Sigma^{-1}O\Psi O'}}=\E{\tr{L^{-1}\Psi}+\tr{X'\nabla_{X}G'}}
\\&\qquad=
\E{
\sum_{\gamma}\frac{\psi_\gamma}{l_\gamma}
+\frac1n\sum_{k,l,\alpha}X'_{k\alpha}\frac{\partial}{\partial X_{\alpha l}}OL^{-1}\Psi O'_{kl}
}.
\end{align*}
But the expression inside the expected value is precisely eq. (\ref{eq:KONN-ure-xpr}), so
\begin{align*}
&\qquad= 
\E{\sum_{k=1}^p \frac{n-p-1}n\frac{\psi_k}{l_k}+\frac{2}{n}\sum_{k=1}^p\frac{\partial \psi_k}{\partial l_k}+\frac1{n}\sum_{k\neq b}^p\frac{\psi_k-\psi_b}{l_k-l_b}}
\end{align*}
as desired.
\end{proof}

\begin{lem}\label{lem:KONN2} Let $W\sim W_p(n,\Sigma)$ with $n\geq p$, and let $W/n=OLO'$ be the spectral decomposition of the associated sample covariance matrix. Let $\psi_1(L),...,\psi_p(L)$ be twice-differentiable functions of the eigenvalues of $W/n$, and define the associated quantities
\begin{align*}
\psi^*_k=\frac{n-p-1}n\frac{\psi_k^2}{l_k}+\frac{4}n\psi_k\frac{\partial \psi_k}{\partial l_k}+\frac2n\psi_k\sum_{b\neq k}^p\frac{\psi_k-\psi_b}{l_k-l_b}
\qquad\text{ for }k=1,...,p
\end{align*} 
with $\Psi^*=\text{diag}(\psi_1,...,\psi_p)$. Assume 
\begin{align*}
&\E{\sum_{k=1}^p\left|\frac{\psi^*_k}{l_k}\right|^{1+\epsilon}}<\infty
\qquad\qquad\text{ and }
\\&
\E{\left|\sum_{k=1}^p\frac{n-p-1}n\frac{\psi_k^*}{l_k}+\frac2n\sum_{k=1}^p\frac{\partial \psi_k^*}{\partial l_k}+\frac1n\sum_{k\neq b}^p\frac{\psi_k^*-\psi_b^*}{l_k-l_b}\right|}<\infty.
\end{align*}
for some $\epsilon>0$. Then
{\setlength{\mathindent}{5pt}\begin{align*}
\E{\!\vphantom{\bigg\vert}\tr{\big[\Sigma^{-1}O\Psi O'\big]^2}\!}
\!=
\E{\sum_{k=1}^p \frac{n-p-1}n\frac{\psi_k^*}{l_k}+\!\frac2n\sum_{k=1}^p\frac{\partial \psi_k^*}{\partial l_k}+\!\frac1n\sum_{k\neq b}^p\frac{\psi_k^*-\psi_b^*}{l_k-l_b}}\!\!.
\end{align*}}
\end{lem}

\begin{proof}[\bf Proof]
We use Lemma 3 in \cite{ChetelatWells12} again. Decompose $W=X'X$ for some $X\sim N_{n\times p}(0,I_n\otimes \Sigma)$, and define $\tilde X=X\Sigma^{-1/2}$ and \linebreak $H=\frac1n\Sigma^{1/2} OL^{-1}\Psi O' \Sigma^{-1}O\Psi O'\Sigma^{-1/2}$. Then
{\setlength{\mathindent}{5pt}\begin{align}
&\text{div}_{\text{vec}(\tilde X)}\;\text{vec}\left(\tilde X H\right)
=
\sum_{\alpha,i,j}\frac{\partial}{\partial \tilde X_{\alpha i}} \left\lbrace \tilde X_{\alpha j} H_{ji} \right\rbrace 
=
n\sum_{i}H_{ii}+\sum_{\alpha,j}\tilde X_{\alpha j} \frac{\partial H_{ji}}{\partial \tilde X_{\alpha i}} 
\notag\\&\quad=
\sum_{i,j,\gamma}\Sigma^{-1}_{ij}O_{j\gamma}\frac{\psi_\gamma^2}{l_\gamma}O'_{\gamma i}
\\&\qquad\qquad
+\frac1n\sum_{\alpha,\beta,i,j,k,l}\tilde X_{\alpha j} \Sigma^{1/2}_{\beta i}\Sigma^{1/2}_{jk} \frac{\partial}{\partial X_{\alpha \beta}} \left\lbrace  OL^{-1}\Psi O' \Sigma^{-1}O\Psi O'\right\rbrace_{kl}\Sigma^{-1/2}_{li}
\notag\\&\quad=
\sum_{i,j,\gamma}\Sigma^{-1}_{ij}O_{j\gamma}\frac{\psi_\gamma^2}{l_\gamma}O'_{\gamma i}
+\frac1n\sum_{\alpha,k,l} X_{\alpha  k}  \frac{\partial}{\partial X_{\alpha l}} OL^{-1}\Psi O'_{k i} \Sigma^{-1}_{ij}O\Psi O'_{j l}
\label{eq:KONN2-ure-xpr}
\\&\quad=
\sum_{i,j,\gamma}\Sigma^{-1}_{ij}O_{j\gamma}\frac{\psi_\gamma^2}{l_\gamma}O'_{\gamma i}
+\frac1n\sum_{i,j,\alpha,k,l,\beta}\Sigma^{-1}_{ij}O\Psi O'_{j l} X_{\alpha  k}  \frac{\partial O_{k\beta}}{\partial X_{\alpha l}}\left[L^{-1}\Psi\right]_{\beta\beta}O'_{\beta i}
\notag\\&\quad\quad
+\frac1n\sum_{i,j,\alpha,k,l,\beta}\Sigma^{-1}_{ij}O\Psi O'_{j l} X_{\alpha  k}  O_{k\beta}\frac{\partial \left[L^{-1}\Psi\right]_{\beta\beta}}{\partial X_{\alpha l}}O'_{\beta i}
\notag\\&\quad\quad
+\frac1n\sum_{i,j,\alpha,k,l,\beta}\Sigma^{-1}_{ij}O\Psi O'_{j l} X_{\alpha  k}  O_{k \beta}\left[L^{-1}\Psi\right]_{\beta\beta}\frac{\partial O_{i \beta}}{\partial X_{\alpha l}}
\notag\\&\quad\quad
+\frac1n\sum_{i,j,\alpha,k,l} X_{\alpha  k}   OL^{-1}\Psi O'_{k i} \Sigma^{-1}_{ij}\frac{\partial O_{j\beta}}{\partial X_{\alpha l}}\Psi_{\beta\beta} O'_{\beta l}
\notag\\&\quad\quad
+\frac1n\sum_{i,j,\alpha,k,l} X_{\alpha  k}   OL^{-1}\Psi O'_{k i} \Sigma^{-1}_{ij}O_{j\beta}\frac{\partial \Psi_{\beta\beta}}{\partial X_{\alpha l}}O'_{\beta l}
\notag\\&\quad\quad
+\frac1n\sum_{i,j,\alpha,k,l} X_{\alpha  k}   OL^{-1}\Psi O'_{k i} \Sigma^{-1}_{ij}O_{j\beta}\Psi_{\beta\beta}\frac{\partial O_{l\beta}}{\partial X_{\alpha l}}
\notag\\&\quad=
\sum_{i,j,\gamma}\Sigma^{-1}_{ij}O_{j\gamma}\frac{\psi_\gamma^2}{l_\gamma}O'_{\gamma i}
\notag\\&\quad\quad
+\frac1{n^2}\hspace{-20pt}\sum_{i,j,k,l,\alpha,\beta,\gamma\neq\beta,\epsilon}\hspace{-20pt}
\Sigma^{-1}_{ij}O\Psi O'_{j l} X_{\alpha  k}  O_{k\gamma}\frac{O'_{\gamma l}O_{\epsilon\beta}+O'_{\gamma\epsilon}O_{l\beta}}{l_\beta-l_\gamma}X_{\alpha\epsilon}
\left[L^{-1}\Psi\right]_{\beta\beta}O'_{\beta i}
\notag\\&\quad\quad
+\frac2{n^2}\hspace{-10pt}\sum_{i,j,\alpha,\beta,\gamma,\epsilon}\hspace{-10pt}
\Sigma^{-1}_{ij}O\Psi O'_{j l} X_{\alpha  k}  O_{k\beta}
O'_{\gamma l}X_{\alpha\epsilon}O_{\epsilon\gamma}
\frac{\partial \left[\psi_\beta/l_\beta\right]}{\partial l_\gamma}
O'_{\beta i}
\notag\\&\quad\quad
+\frac1{n^2}\hspace{-18pt}\sum_{i,j,k,l,\alpha,\beta,\gamma\neq\beta,\epsilon}\hspace{-18pt}
\Sigma^{-1}_{ij}O\Psi O'_{j l} X_{\alpha  k}  O_{k \beta}\left[L^{-1}\Psi\right]_{\beta\beta}O_{i\gamma}\frac{O'_{\gamma l}O_{\epsilon\beta}+O'_{\gamma\epsilon}O_{l\beta}}{l_\beta-l_\gamma}X_{\alpha\epsilon}
\notag\\&\quad\quad
+\frac1{n^2}\hspace{-15pt}\sum_{i,j,\alpha,\beta,\gamma\neq\beta,\epsilon}\hspace{-20pt}
 X_{\alpha  k}   OL^{-1}\Psi O'_{k i} \Sigma^{-1}_{ij}
O_{j\gamma}\frac{O'_{\gamma l}O_{\epsilon\beta}+O'_{\gamma\epsilon}O_{l\beta}}{l_\beta-l_\gamma}X_{\alpha\epsilon}
\Psi_{\beta\beta} O'_{\beta l}
\notag\\&\quad\quad
+\frac2{n^2}\hspace{-10pt}\sum_{i,j,\alpha,\beta,\gamma,\epsilon}\hspace{-10pt}
 X_{\alpha  k}   OL^{-1}\Psi O'_{k i} \Sigma^{-1}_{ij}O_{j\beta}
O'_{\gamma l}X_{\alpha\epsilon}O_{\epsilon\gamma}
\frac{\partial \psi_\beta}{\partial l_\gamma}
O'_{\beta l}
\notag\\&\quad\quad
+\frac1{n^2}\hspace{-18pt}\sum_{i,j,k,l,\alpha,\beta,\gamma\neq\beta,\epsilon}\hspace{-18pt}
 X_{\alpha  k}   OL^{-1}\Psi O'_{k i} \Sigma^{-1}_{ij}O_{j\beta}\Psi_{\beta\beta}
O_{l\gamma}\frac{O'_{\gamma l}O_{\epsilon\beta}+O'_{\gamma\epsilon}O_{l\beta}}{l_\beta-l_\gamma}X_{\alpha\epsilon}
\notag\\&\quad=
\sum_{i,j,\gamma}\Sigma^{-1}_{ij}O_{j\gamma}\frac{\psi_\gamma^2}{l_\gamma}O'_{\gamma i}
\notag\\&\quad\quad
+\frac1{n}\hspace{-10pt}\sum_{i,j,l,\beta,\gamma\neq\beta}\hspace{-10pt}
\Sigma^{-1}_{ij}
O_{j\beta}\frac{l_\gamma\psi_\beta^2}{(l_\beta-l_\gamma)l_\beta}O'_{\beta i}
+\frac2{n}\sum_{i,j,\gamma}
\Sigma^{-1}_{ij}O_{j\gamma}\psi_\gamma l_\gamma
\frac{\partial \left[\psi_\gamma/l_\gamma\right]}{\partial l_\gamma}O'_{\gamma i}
\notag\\&\quad\quad
+\frac1{n}\hspace{-5pt}\sum_{i,j,\beta,\gamma\neq\beta}\hspace{-5pt}
\Sigma^{-1}_{ij}O_{j\gamma}\frac{\psi_\beta\psi_\gamma}{l_\beta-l_\gamma}O'_{\gamma i}
+\frac1n\hspace{-5pt}\sum_{i,j,\beta,\gamma\neq\beta}\hspace{-5pt}
\Sigma^{-1}_{ij}
O_{j\gamma}\frac{\psi_\gamma\psi_\beta}{l_\beta-l_\gamma}O'_{\gamma i} 
\notag\\&\quad\quad
+\frac2n\sum_{i,j,\gamma}
 \Sigma^{-1}_{ij}O_{j\gamma} \psi_\gamma\frac{\partial \psi_\gamma}{\partial l_\gamma}
O'_{\gamma i} 
+\frac1n\hspace{-5pt}\sum_{i,j,\beta,\gamma\neq\beta,\epsilon}\hspace{-5pt}
\Sigma^{-1}_{ij}O_{j\beta}\frac{\psi_\beta^2}{l_\beta-l_\gamma}O'_{\beta i}
\notag\\&\quad=
\frac{n-p-1}{n}\sum_{i,j,\gamma}\Sigma^{-1}_{ij}O_{j\gamma}\frac{\psi_\gamma^2}{l_\gamma}O'_{\gamma i}
+\frac4{n}\sum_{i,j,\gamma}
\Sigma^{-1}_{ij}O_{j\gamma}\psi_\gamma 
\frac{\partial \psi_\gamma}{\partial l_\gamma}O'_{\gamma i}
\notag\\&\quad\quad
+\frac2{n}\hspace{-10pt}\sum_{i,j,l,\beta,\gamma\neq\beta}\hspace{-10pt}
\Sigma^{-1}_{ij}
O_{j\gamma}
\psi_\gamma\frac{(\psi_\gamma-\psi_\beta)}{(l_\gamma-l_\beta)}
O'_{\gamma i}
\end{align}}
Thus
\begin{align*}
&\E{\left|\text{div}_{\text{vec}(\tilde X)}\;\text{vec}\left(\tilde X H\right)\right|}
=
\frac1n\E{\left|\sum_{i,j=1}^p\Sigma^{-1}_{ij}O\Psi^*O'_{ji}\right|}
\\&\qquad\leq
\frac1n\E{\sum_{k=1}^p
\bigg|\left[L^{1/2}O'\Sigma^{-1}OL^{1/2}\right]_{kk}\bigg|
\left|\frac{\psi^*_k}{l_k}\right|
}
\\&\qquad\leq
\frac1n
\sum_{k=1}^p\E{\left[L^{1/2}O'\Sigma^{-1}OL^{1/2}\right]^{1+\frac1\epsilon}_{kk}}^{\frac\epsilon{1+\epsilon}}
\E{\left|\frac{\psi^*_k}{l_k}\right|^{1+\epsilon}}^{\frac1{1+\epsilon}}
\\&\qquad\leq
\frac1n
\left(\E{\sum_{k=1}^p\left[L^{1/2}O'\Sigma^{-1}OL^{1/2}\right]
^{1+\frac1\epsilon}_{kk}}\right)^{\frac\epsilon{1+\epsilon}}
\left(\E{\sum_{k=1}^p\left|\frac{\psi^*_k}{l_k}\right|^{1+\epsilon}}\right)^{\frac1{1+\epsilon}}
\end{align*}
Since 
\begin{align*}
&\sum_{k=1}^p\left[L^{1/2}O'\Sigma^{-1}OL^{1/2}\right]^{1+\frac1\epsilon}_{kk}
\leq 
\left(\sum_{k=1}^p\left|L^{1/2}O'\Sigma^{-1}OL^{1/2}\right|_{kk}\right)^{1+\frac1\epsilon}
\\&\qquad=
\tr{L^{1/2}O'\Sigma^{-1}OL^{1/2}}^{1+\frac1\epsilon}
=
\tr{\Sigma^{-1}S}^{1+\frac1\epsilon}
\sim 
(\chi^2_{pn})^{1+\frac1\epsilon}
\end{align*}
we get
\begin{align*}
&\E{\left|\text{div}_{\text{vec}(\tilde X)}\;\text{vec}\left(\tilde X H\right)\right|}
\\&\;\;\leq
\frac{2\Gamma\left(1+\frac1\epsilon+\frac{np}2\right)^{\frac\epsilon{1+\epsilon}}}{n\Gamma\left(\frac{np}2\right)^{\frac\epsilon{1+\epsilon}}}
\left(\E{\sum_{k=1}^p\left|\frac{\psi^*_k}{l_k}\right|^{1+\epsilon}}\right)^{\frac1{1+\epsilon}}
<\infty
\end{align*}
by assumption of the lemma. Therefore by Lemma 3 in \cite{ChetelatWells12},
\begin{align*}
&\E{\tr{\big[\Sigma^{-1}O\Psi O'\big]^2}}=\E{\tr{L^{-1}\Psi}+\tr{X'\nabla_{X}G'}}
\\&\quad=
\E{
\sum_{i,j,k=1}^p\Sigma^{-1}_{ij}O_{jk}\frac{\psi_k^2}{l_k}O'_{k i}
\right.\\&\qquad\qquad\left.
+\frac1n\sum_{\alpha=1}^n\sum_{k,l=1}^p X_{\alpha  k}  \frac{\partial}{\partial X_{\alpha l}} OL^{-1}\Psi O'_{k i} \Sigma^{-1}_{ij}O\Psi O'_{j l}}
\\&\quad=
\E{\sum_{i,j,k=1}^p\Sigma^{-1}_{ij}O_{jk}\psi^*_k O'_{k i}}
\qquad\qquad\big(\text{ by (\ref{eq:KONN2-ure-xpr}) }\big).
\end{align*}
Finally, by Lemma \ref{lem:KONN}, we conclude
\begin{align*}
&\quad=
\E{\sum_{k=1}^p \frac{n-p-1}n\frac{\psi_k^*}{l_k}+\frac{2}n\sum_{k=1}^p\frac{\partial \psi_k^*}{\partial l_k}+\frac1n\sum_{k\neq b}^p\frac{\psi_k^*-\psi_b^*}{l_k-l_b}}
\end{align*}
as desired.
\end{proof}

\begin{lem}\label{lem:BDFC} Let $l_1>...>l_p>0$ be the eigenvalues of a $W_p(n,\Sigma)$-distributed matrix ,for some $\Sigma>0$. If $n\geq p+1$, then
\begin{enumerate}[(i)]
\item for any $1\leq k\leq p$ and $0\leq m<\frac{n-p-1}2$, $\E{\frac1{|l_k|^m}}<\infty$;
\item for any $1\leq k\neq b\leq p$ and $1\leq m<2$, $\E{\frac1{|l_k-l_b|^m}}<\infty$;
\item for any $1\leq k\neq b\neq e\leq p$ and $1\leq m<2$, $\E{\frac1{|l_k-l_b|^m|l_k-l_e|^m}}<\infty$.
\end{enumerate}
\end{lem}

\begin{proof}[\bf Proof] First, notice that in (ii), we can take $k<b$ without loss of generality. Then
\begin{align*}
\E{\frac1{|l_k-l_b|^m}}\leq \E{\frac1{|l_{k}-l_{k+1}|^m}},
\end{align*}
and it would be enough to show the r.h.s. finite for all $1\leq k< p$ to show (ii).

Similarly, in (iii), we can take $b<e$ without loss of generality, and there are then three possibilities. Either $k<b<e$, in which case
\begin{align*}
\E{\frac1{|l_k-l_b|^m|l_k-l_e|^m}}\leq \E{\frac1{|l_{k}-l_{k+1}|^m|l_{e-1}-l_e|^m}},
\end{align*}
or $b<k<e$ in which case
\begin{align*}
\E{\frac1{|l_k-l_b|^m|l_k-l_e|^m}}\leq \E{\frac1{|l_b-l_{b+1}|^m|l_{e-1}-l_e|^m}}
\end{align*}
or $b<e<k$ in which case
\begin{align*}
\E{\frac1{|l_k-l_b|^m|l_k-l_e|^m}}\leq \E{\frac1{|l_b-l_{b+1}|^m|l_e-l_{e+1}|^m}}
\end{align*}
Thus in any case it is enough to show that
\begin{align*}
\E{\frac1{|l_k-l_{k+1}|^m|l_b-l_{b+1}|^m}}<\infty
\end{align*}
for all $1\leq k<b< p$ to show (iii).

By \cite{Muirhead82}, Theorem 3.2.18, the joint density of $l_1>...>l_p$ is given by
\begin{align}
f_{l_1,...,l_p}(l_1,...,l_p)=&\;\frac{\pi^{p^2/2}2^{-pn}|\Sigma|^{-n/2}}{\Gamma_p(p/2)\Gamma_p(n/2)}\prod_{i=1}^pl_i^{\frac{n-p-1}2}\prod_{1\leq i<j\leq p}(l_i-l_j)
\\&\;\int_{O(p)}\text{etr}\left(-\frac 12\Sigma^{-1}HLH'\right)dH\;\1{l_1>...>l_p>0}
\label{eq:BDFC-density}
\end{align}
for $L=\text{diag}(l_1,...,l_p)$. Define $I_2=\lbrace (i,j)\,\vert\, i<j \wedge (i,j)\neq (k,k+1) \rbrace$ and $I_3=\lbrace (i,j)\,\vert\, i<j \wedge (i,j)\neq (k,k+1), (b,b+1)\rbrace$. The expressions
\begin{align*}
&P_1(l_1,...,l_p)=\prod_{i\neq k}^p l_i^{n-p-1}\prod_{i<j}^p(l_i-l_j)^2
\\&
P_2(l_1,...,l_p)=\prod_{i=1}^pl_i^{n-p-1}\prod_{(i,j)\in I_2}(l_i-l_j)^2
\\&
P_3(l_1,...,l_p)=\prod_{i=1}^pl_i^{n-p-1}\prod_{(i,j)\in I_3}(l_i-l_j)^2
\end{align*}
and $K=\frac{\pi^{p^2/2}2^{-pn}|\Sigma|^{-n/2}}{\Gamma_p(p/2)\Gamma_p(n/2)}$ can then be defined to write
{\setlength{\mathindent}{0pt}\begin{align*}
&f_{l_1,...,l_p}(l_1,...,l_p)
=
K l_k^{\frac{n-p-1}2} P^{1/2}_1(l_1,...,l_p)
\\&\qquad\qquad
\int_{O(p)}\!\!\!\!\!\!\text{etr}\left(-\frac 12\Sigma^{-1}HLH'\right)dH\;\1{l_1>...>l_p>0}
\\&\qquad
=K |l_k-l_{k+1}| P^{1/2}_2(l_1,...,l_p)
\\&\qquad\qquad
\int_{O(p)}\!\!\!\!\!\!\text{etr}\left(-\frac 12\Sigma^{-1}HLH'\right)dH\;\1{l_1>...>l_p>0}
\\&\qquad=
K |l_k-l_{k+1}||l_b-l_{b+1}| P^{1/2}_3(l_1,...,l_p)
\\&\qquad\qquad
\int_{O(p)}\!\!\!\!\!\!\text{etr}\left(-\frac 12\Sigma^{-1}HLH'\right)dH\;\1{l_1>...>l_p>0}.
\end{align*}}
The important point is that since $n-p-1\geq 0$, $P_1$, $P_2$ and $P_3$ are polynomials in $l_1,...,l_p$. Define $x=l_{k+1}-l_k$, $y=l_{k+1}-l_b$ and $z=l_b-l_{b+1}$, so that $l_k=l_{b+1}+z+y+x$, $l_{k+1}=l_{b+1}+z+y$ and $l_b=l_{b+1}+z$. (It might happen that $k+1$=$b$, something which should be kept in mind.) Then 
{\setlength{\mathindent}{5pt}\begin{align*}
&P_2(l_1,...,l_{k+1}+x,...,l_{k+1},...,l_p),
\\&
P_3(l_1,...,l_{b+1}+z+x,...,l_{b+1}+z,...,l_b)
&\text{ if }k+1=b,
\\&
P_3(l_1,...,l_{b+1}+z+y+x,...,l_{b+1}+z+y,...,l_{b+1}+z,...,l_b)
&\text{ if }k+1\neq b,
\end{align*}}
must still be polynomials, in $\lbrace l_i\rbrace\backslash\lbrace l_k\rbrace\cup\lbrace x\rbrace$, $\lbrace l_i\rbrace\backslash\lbrace l_k,l_b\rbrace\cup\lbrace x,z\rbrace$ and \linebreak$\lbrace l_i\rbrace\backslash\lbrace l_k,l_{k+1},l_b\rbrace\cup\lbrace x,y,z\rbrace$ respectively. Therefore, for some finite degrees $D_1, ..., D_4$ one can write
\begin{align*}
&P_1(l_1,...,l_p)=
\sum_{\substack{d_1+...+d_p\\ \leq D_1}}
A^1_{d_1,...,d_p}l_1^{d_1}\cdots l_p^{d_p},
\\&
P_2(l_1,...,l_{k+1}+x,...,l_{k+1},...,l_p)=
\sum_{\substack{d_1+...+d_p\\ \leq D_2}}
A^2_{d_1,...,d_p}l_1^{d_1}\cdots x^{d_k}\cdots l_p^{d_p},
\\&
P_3(l_1,...,l_{b+1}+z+x,...,l_{b+1}+z,...,l_b)
\\&\qquad=
\sum_{\substack{d_1+...+d_p\\ \leq D_3}}
A^3_{d_1,...,d_p}l_1^{d_1}\cdots x^{d_k}\cdots z^{d_{b}}\cdots l_p^{d_p},
\qquad\qquad\text{ if }k+1=b,
\\&
P_3(l_1,...,l_{b+1}+z+y+x,...,l_{b+1}+z+y,...,l_{b+1}+z,...,l_b)
\\&\qquad=
\sum_{\substack{d_1+...+d_p\\ \leq D_4}}
A^4_{d_1,...,d_p}l_1^{d_1}\cdots x^{d_k}\cdots y^{d_{k+1}}\cdots z^{d_b}\cdots l_p^{d_p},
\quad\text{ if }k+1\neq b.
\end{align*}
for coefficients $A^1_{d_1,...,d_p}, ..., A^4_{d_1,...,d_p}\in\mathbb{R}$.

If we denote the greatest eigenvalue of $\Sigma$ by $\lambda_{\text{max}}$, then $\Sigma^{-1}\geq \lambda_{\text{max}}^{-1}I$ so that
\begin{align*}
\int_{O(p)}\text{etr}\left(-\frac 12\Sigma^{-1}HLH'\right)dH
\leq
\exp\left(-\frac1{2\lambda_\text{max}}\sum_{i=1}^p l_i\right)
\end{align*}
for any $l_1,...,l_p\geq 0$.

Now, for (i), we can use (\ref{eq:BDFC-density}) and $\1{l_1>...>l_p>0}\leq \prod_{i=1}^p\1{l_i>0}$ to find
\begin{align*}
&\E{\frac1{|l_k|^m}}
\leq
K\int_{\mathbb{R}^p}\frac1{l_k^{m-\frac{n-p-1}2}} P^{1/2}_1(l_1,...,l_p)
\exp\left(-\frac1{2\lambda_\text{max}}\sum_{i=1}^p l_i\right)\;
\\&\qquad\qquad\qquad\qquad
\1{l_1>...>l_p>0}dl_1\cdots dl_p
\\&\qquad\leq
K\sum_{\substack{d_1+...+d_p\\ \leq D_1}}
\sqrt{\left|A^1_{d_1,...,d_p}\right|}
\int_0^\infty l_1^{d_1/2}e^{-l_1/{2\lambda_\text{max}}}\;dl_1 \cdots
\\&\qquad\qquad
\int_0^\infty l_k^{d_k/2-m+\frac{n-p-1}2}e^{-l_k/{2\lambda_\text{max}}}\;dl_k \cdots
\int_0^\infty l_p^{d_p/2}e^{-l_p/{2\lambda_\text{max}}}\;dl_p.
\end{align*}
Notice that for $i\neq k$, $\int_0^\infty l_i^{d_i/2}e^{-l_i/{2\lambda_\text{max}}}\;dl_i<\infty$ for any $d_i\geq 0$, and \linebreak$\int_0^\infty l_k^{d_k/2-m+\frac{n-p-1}2}e^{-l_k/{2\lambda_\text{max}}}\;dl_k<\infty$ for all $d_k\geq 0$ iff $0\leq m<\frac{n-p-1}2$. Thus $\E{1/|l_k|^m}<\infty$, as desired.

For (ii), we proceed similarly, but though a change of variables $(l_k,l_{k+1})\rightarrow (l_{k+1}+x,l_k)$. Then, using 
{\setlength{\mathindent}{5pt}\begin{align*}
\1{\vphantom{\bigg|}l_1>...>l_{k+1}+x>...>l_{k+1}>...>l_p>0}
\leq 
\1{x>0}\prod_{i\neq k}^p\1{l_i>0},
\end{align*}}
we obtain
{\setlength{\mathindent}{5pt}\begin{align*}
&\E{\frac1{|l_k-l_{k+1}|^m}}
\leq
K\int_{\mathbb{R}^p}\frac1{|l_k-l_{k+1}|^{m-1}} P^{1/2}_2(l_1,...,l_p)
\\&\qquad\qquad\qquad\qquad
\exp\left(-\frac1{2\lambda_\text{max}}\sum_{i=1}^p l_i\right)\;
\1{l_1>...>l_p>0}dl_1\cdots dl_p
\\&\qquad\leq
K\sum_{\substack{d_1+...+d_p\\ \leq D_2}}
\sqrt{\left|A^2_{d_1,...,d_p}\right|}
\int_0^\infty l_1^{d_1/2}e^{-l_1/{2\lambda_\text{max}}}\;dl_1 \cdots
\\&\qquad\qquad
\int_0^\infty x^{d_k/2-m+1}e^{-x/{2\lambda_\text{max}}}\;dx \cdots
\int_0^\infty l_{k+1}^{d_{k+1}/2}e^{-l_{k+1}/{\lambda_\text{max}}}\;dl_{k+1} \cdots
\\&\qquad\qquad
\int_0^\infty l_p^{d_p/2}e^{-l_p/{2\lambda_\text{max}}}\;dl_p.
\end{align*}}
Then again, for $i\neq k$ and any $d_i\geq 0$, the respective integrals are finite; and to have $\int_0^\infty x^{d_k/2-m+1}e^{-x/{2\lambda_\text{max}}}\;dx<\infty$ for all $d_k\geq 0$ requires $m<2$. In such a case, we end up with $\E{1/{|l_k-l_{k+1}|^m}}<\infty$, as desired.

For (iii), we must consider separately the cases $k+1=b$ and $k+1\neq b$. In the first case, one can take the change of variables $(l_k,l_b,l_{b+1})\longrightarrow(l_{b+1}+x+z,l_{b+1}+z,l_{b+1})$. Using that
\begin{align*}
&\1{\vphantom{\bigg|}l_1>...>l_{b+1}+x+z>...>l_{b+1}+z>...>l_{b+1}>...>l_p>0}
\\&\qquad\qquad\leq 
\1{x>0}\1{z>0}\prod_{i\neq k,b}^p\1{l_i>0},
\end{align*}
we then obtain
{\setlength{\mathindent}{5pt}\begin{align*}
&\E{\frac1{|l_k-l_{b}|^m|l_b-l_{b+1}|^m}}
\leq
K\int_{\mathbb{R}^p}\frac1{|l_k-l_{b}|^{m-1}|l_b-l_{b+1}|^{m-1}} 
\\&\qquad\qquad
P^{1/2}_3(l_1,...,l_p)
\exp\left(-\frac1{2\lambda_\text{max}}\sum_{i=1}^p l_i\right)\;
\1{l_1>...>l_p>0}dl_1\cdots dl_p
\\&\qquad\leq
K\sum_{\substack{d_1+...+d_p\\ \leq D_3}}
\sqrt{\left|A^3_{d_1,...,d_p}\right|}
\int_0^\infty l_1^{d_1/2}e^{-l_1/{2\lambda_\text{max}}}\;dl_1 \cdots
\\&\qquad\qquad
\int_0^\infty x^{d_k/2-m+1}e^{-x/{2\lambda_\text{max}}}\;dx \cdots
\int_0^\infty z^{d_b/2-m+1}e^{-z/{\lambda_\text{max}}}\;dz \cdots
\\&\qquad\qquad
\int_0^\infty l_{b+1}^{d_{b+1}/2}e^{-3l_{b+1}/{2\lambda_\text{max}}}\;dl_{b+1} \cdots
\int_0^\infty l_p^{d_p/2}e^{-l_p/{2\lambda_\text{max}}}\;dl_p.
\end{align*}}
Again, all the integrals converge as long as $m<2$, in which case we have $\E{1/{|l_k-l_b|^m|l_b-l_{b+1}|^m}}<\infty$, as desired.

Finally, for (iii) with $k+1\neq b$, one can take the change of variables $(l_k,l_{k+1},l_b,l_{b+1})\longrightarrow
(l_{b+1}+x+y+z,l_{b+1}+y+z,l_{b+1}+z,l_{b+1})$. Then using that
\begin{align*}
&\1{\vphantom{\bigg|}
l_1>...>l_{b+1}+x+y+z>...>l_{b+1}+y+z>...>l_{b+1}+z
\right.\\&\qquad\qquad\left.
>...>l_{b+1}>...>l_p>0\vphantom{\bigg|}}
\\&\qquad\qquad\leq 
\1{x>0}\1{y>0}\1{z>0}\prod_{i\neq k,k+1,b}^p\1{l_i>0},
\end{align*}
we obtain
{\setlength{\mathindent}{5pt}\begin{align*}
&\E{\frac1{|l_k-l_{k+1}|^m|l_b-l_{b+1}|^m}}
\leq
K\int_{\mathbb{R}^p}\frac1{|l_k-l_{k+1}|^{m-1}|l_b-l_{b+1}|^{m-1}} 
\\&\qquad\qquad
P^{1/2}_4(l_1,...,l_p)
\exp\left(-\frac1{2\lambda_\text{max}}\sum_{i=1}^p l_i\right)\;
\1{l_1>...>l_p>0}dl_1\cdots dl_p
\\&\qquad\leq
K\sum_{\substack{d_1+...+d_p\\ \leq D_4}}
\sqrt{\left|A^4_{d_1,...,d_p}\right|}
\int_0^\infty l_1^{d_1/2}e^{-l_1/{2\lambda_\text{max}}}\;dl_1 \cdots
\\&\qquad\qquad
\int_0^\infty x^{d_k/2-m+1}e^{-x/{2\lambda_\text{max}}}\;dx \cdots
\int_0^\infty y^{d_k/2-m+1}e^{-y/{\lambda_\text{max}}}\;dy \cdots
\\&\qquad\qquad
\int_0^\infty z^{d_b/2-m+1}e^{-3z/{2\lambda_\text{max}}}\;dz \cdots
\int_0^\infty l_{b+1}^{d_{b+1}/2}e^{-2l_{b+1}/{\lambda_\text{max}}}\;dl_{b+1} \cdots
\\&\qquad\qquad
\int_0^\infty l_p^{d_p/2}e^{-l_p/{2\lambda_\text{max}}}\;dl_p.
\end{align*}}
All the integrals converge as long as $m<2$, in which case we have \linebreak$\E{1/{|l_k-l_{b+1}|^m|l_b-l_{b+1}|^m}}<\infty$, as desired.
\end{proof}

\begin{proof}[\bf Proof of Theorem \ref{thm:URE}]
By independence, it is clear that
\begin{align*}
\E{L_H(\hat\Sigma,\Sigma)}
&=\text{E}_{\hat\rho}\left[\text{E}_S\left[L_H(\hat\Sigma,\Sigma)\,\vert\,\hat\rho\right]\right]
\\
&=\text{E}_{\hat\rho}\left[\text{E}_S\left[L_H(\hat\Sigma,\Sigma)\right]\right],
\end{align*}
so we can treat $\hat\rho$ as a constant throughout the calculations, without loss of generality. Define the auxiliary terms $\psi_k=\hat\gamma_k+\hat\sigma^2$ and 
\begin{align*}
\psi^*_k=\frac{n-p-1}n\frac{\psi^{2}_k}{l_k}+4\frac{\psi_k}n\frac{\partial \psi_k}{\partial l_k}+2\frac{\psi_k}n\sum_{b\neq k}^p\frac{\psi_k-\psi_b}{l_k-l_b}
\end{align*}
for all $1\leq k\leq p$, and consider:
{\setlength{\mathindent}{0pt}\begin{align*}
&R_1=\sum_{k=1}^p\frac{n-p-1}n\frac{\psi_k^{*}}{l_k}+\frac2n\sum_{k=1}^p\frac{\partial \psi_k^{*}}{\partial l_k}+\frac1n\sum_{k\neq b}^p\frac{\psi_k^{*}-\psi_b^{*}}{l_k-l_b}
\\&\;\;=
\left\lbrace
\frac{(n-p-1)^2}{n^2}\sum_{k=1}^{p}\frac{\psi^2_k}{l_k^2}
+4\frac{n-p-1}{n^2}\sum_{k=1}^{p}\frac{\psi_k}{l_k}\frac{\partial\psi_k}{\partial l_k}
\right.\\&\qquad\left.
+2\frac{n-p-1}{n^2}\sum_{k\neq b=1}^p\frac{\psi_k}{l_k}\frac{\psi_k-\psi_b}{l_k-l_b}
\right\rbrace
+\left\lbrace
4\frac{n-p-1}{n^2}\sum_{k=1}^{p}\frac{\psi_k}{l_k}\frac{\partial\psi_k}{\partial l_k}
\right.\\&\qquad\left.
-2\frac{n-p-1}{n^2}\sum_{k=1}^{p}\frac{\psi^2_k}{l_k^2}
+\frac8{n^2}\sum_{k=1}^{p}\left(\frac{\partial\psi_k}{\partial l_k}\right)^2
+\frac8{n^2}\sum_{k=1}^{p}\psi_k\frac{\partial^2\psi_k}{\partial l_k^2}
\right.\\&\qquad\left.
+\frac4{n^2}\sum_{k\neq b=1}^p\frac{\partial\psi_k}{\partial l_k}\frac{\psi_k-\psi_b}{l_k-l_b}
+\frac4{n^2}\sum_{k\neq b=1}^p\psi_k\frac{\frac{\partial\psi_k}{\partial l_k}-\frac{\partial\psi_b}{\partial l_k}}{l_k-l_b}
\right.\\&\qquad\left.
-\frac2{n^2}\sum_{k\neq b=1}^p\left(\frac{\psi_k-\psi_b}{l_k-l_b}\right)^2
\right\rbrace
+\left\lbrace
\left(
2\frac{n-p-1}{n^2}\sum_{k\neq b=1}^p\frac{\psi_k}{l_k}\frac{\psi_k-\psi_b}{l_k-l_b}
\right.\right.\\&\qquad\left.\left.
-\frac{n-p-1}{n^2}\sum_{k\neq b=1}^p\frac{\psi_k}{l_k}\frac{\psi_b}{l_b}
\right)
+\left(
\frac4{n^2}\sum_{k\neq b=1}^p\psi_k\frac{\frac{\partial\psi_k}{\partial l_k}-\frac{\partial\psi_b}{\partial l_b}}{l_k-l_b}
\right.\right.\\&\qquad\left.\left.
+\frac4{n^2}\sum_{k\neq b=1}^p\frac{\partial\psi_k}{\partial l_k}\frac{\psi_k-\psi_b}{l_k-l_b}
\right)
+\left(
\frac2{n^2}\!\!\sum_{k\neq b\neq e=1}^p\frac{\psi_k}{l_k-l_b}\left(\frac{\psi_k-\psi_e}{l_k-l_e}-\frac{\psi_b-\psi_e}{l_b-l_e}\right)
\right.\right.\\&\qquad\left.\left.
+\frac2{n^2}\sum_{k\neq b\neq e=1}^p\frac{\psi_k-\psi_b}{l_k-l_b}\frac{\psi_k-\psi_e}{l_k-l_e}
+\frac2{n^2}\sum_{k\neq b=1}^p\left(\frac{\psi_k-\psi_b}{l_k-l_b}\right)^2
\right)
\right\rbrace
\\&\;\;=
\frac{(n-p-1)(n-p-2)}{n^2}\sum_{k=1}^{p}\frac{\psi^2_k}{l_k^2}
-\frac{(n-p-1)}{n^2}\left(\sum_{k=1}^p\frac{\psi_k}{l_k}\right)^2
\\&\qquad
+\frac8{n^2}\sum_{k=1}^{p}\left(\frac{\partial\psi_k}{\partial l_k}\right)^2
+\frac8{n^2}\sum_{k=1}^{p}\psi_k\frac{\partial^2\psi_k}{\partial l_k^2}
+8\frac{n-p-1}{n^2}\sum_{k=1}^{p}\frac{\psi_k}{l_k}\frac{\partial\psi_k}{\partial l_k}
\\&\qquad
+4\frac{n-p-1}{n^2}\sum_{k\neq b=1}^p\frac{\psi_k}{l_k}\frac{\psi_k-\psi_b}{l_k-l_b}
+\frac8{n^2}\sum_{k\neq b=1}^p\frac{\partial\psi_k}{\partial l_k}\frac{\psi_k-\psi_b}{l_k-l_b}
\\&\qquad
+\frac4{n^2}\sum_{k\neq b=1}^p\psi_k\frac{\frac{\partial\psi_k}{\partial l_k}-\frac{\partial\psi_b}{\partial l_b}}{l_k-l_b}
+\frac4{n^2}\sum_{k\neq b=1}^p\psi_k\frac{\frac{\partial\psi_k}{\partial l_k}-\frac{\partial\psi_b}{\partial l_k}}{l_k-l_b}
\\&\qquad
+\frac2{n^2}\sum_{k\neq b\neq e=1}^p\frac{\psi_k}{l_k-l_b}\left(\frac{\psi_k-\psi_e}{l_k-l_e}-\frac{\psi_b-\psi_e}{l_b-l_e}\right)
\\&\qquad
+\frac2{n^2}\sum_{k\neq b\neq e=1}^p\frac{\psi_k-\psi_b}{l_k-l_b}\frac{\psi_k-\psi_e}{l_k-l_e}.
\end{align*}}
Now, by H\"older's inequality, we find:
{\setlength{\mathindent}{0pt}\begin{align*}
&\;\;\E{\vphantom{\bigg\vert}|R_1|}
\leq
\frac{|n-p-1||n-p-2|}{n^2}\sum_{k=1}^{p}\E{\left|\frac{\psi_k}{l_k}\right|^2}
\\&\qquad
+\frac{|n-p-1|}{n^2}\left(\sum_{k=1}^p\E{\left|\frac{\psi_k}{l_k}\right|^2}^{\frac12}\right)^2
+\frac8{n^2}\sum_{k=1}^{p}\E{\left|\frac{\partial\psi_k}{\partial l_k}\right|^2}
\\&\qquad
+\frac8{n^2}\sum_{k=1}^{p}\E{\left|\psi_k\frac{\partial^2\psi_k}{\partial l_k^2}\right|}
+8\frac{|n-p-1|}{n^2}\sum_{k=1}^{p}\E{\left|\frac{\psi_k}{l_k}\right|^2}^\frac12\E{\left|\frac{\partial\psi_k}{\partial l_k}\right|^2}^\frac12
\\&\qquad
+4\frac{|n-p-1|}{n^2}\sum_{k\neq b=1}^p\E{\left|\frac{\psi_k}{l_k}\right|^{4.5}}^\frac1{4.5}
\\&\hspace{120pt}
\left(\E{\left|\psi_k\right|^{4.5}}^\frac1{4.5}
+
\E{\left|\psi_b\right|^{4.5}}^\frac1{4.5}
\right)
\E{\frac1{|l_k-l_b|^{1.8}}}^\frac1{1.8}
\\&\qquad
+\frac8{n^2}\sum_{k\neq b=1}^p\left(\E{\left|\frac{\partial\psi_k}{\partial l_k}\right|^{4.5}}^\frac1{4.5}\E{\left|\psi_k\right|^{4.5}}^\frac1{4.5}
\right.\\&\hspace{120pt}\left.
+
\E{\left|\frac{\partial\psi_k}{\partial l_k}\right|^{4.5}}^\frac1{4.5}\E{\left|\psi_b\right|^{4.5}}^\frac1{4.5}\right)\E{\frac1{|l_k-l_b|^{1.8}}}^\frac1{1.8}
\\&\qquad
+\frac4{n^2}\sum_{k\neq b=1}^p\E{\left|\psi_k\right|^{4.5}}^\frac1{4.5}
\left(
\E{\left|\frac{\partial\psi_k}{\partial l_k}\right|^{4.5}}^\frac1{4.5}
\right.\\&\hspace{180pt}\left.
+\E{\left|\frac{\partial\psi_b}{\partial l_b}\right|^{4.5}}^\frac1{4.5}
\right)
\E{\frac1{|l_k-l_b|^{1.8}}}^\frac1{1.8}
\\&\qquad
+\frac4{n^2}\sum_{k\neq b=1}^p\E{\left|\psi_k\right|^{4.5}}^\frac1{4.5}
\left(
\E{\left|\frac{\partial\psi_k}{\partial l_k}\right|^{4.5}}^\frac1{4.5}
\right.\\&\hspace{180pt}\left.
+\E{\left|\frac{\partial\psi_b}{\partial l_k}\right|^{4.5}}^\frac1{4.5}
\right)
\E{\frac1{|l_k-l_b|^{1.8}}}^\frac1{1.8}
\\&\qquad
+\frac2{n^2}\sum_{k\neq b\neq e=1}^p\left(\E{\left|\psi_k\right|^{4.5}}^\frac1{4.5}+\E{\left|\psi_b\right|^{4.5}}^\frac1{4.5}\right)
\\&\hspace{75pt}
\left(\E{\left|\psi_k\right|^{4.5}}^\frac1{4.5}+\E{\left|\psi_e\right|^{4.5}}^\frac1{4.5}\right)
\E{\left|\frac1{(l_k-l_b)(l_k-l_e)}\right|^{1.8}}^\frac1{1.8}
\\&\qquad
+\frac2{n^2}\sum_{k\neq b\neq e=1}^p\left(\E{\left|\psi_k\right|^{4.5}}^\frac1{4.5}+\E{\left|\psi_b\right|^{4.5}}^\frac1{4.5}\right)
\\&\hspace{75pt}
\left(\E{\left|\psi_k\right|^{4.5}}^\frac1{4.5}+\E{\left|\psi_e\right|^{4.5}}^\frac1{4.5}\right)
\E{\left|\frac1{(l_k-l_b)(l_k-l_e)}\right|^{1.8}}^\frac1{1.8}.
\end{align*}}
Similarly, consider
\begin{align*}
R_2=\frac{n-p-1}n\sum_{k=1}^{p}\frac{\psi_k}{l_k}
+\frac2n\sum_{k=1}^{p}\frac{\partial\psi_k}{\partial l_k}
+\frac1n\sum_{k\neq b=1}^p\frac{\psi_k-\psi_b}{l_k-l_b},
\end{align*}
so that
{\setlength{\mathindent}{0pt}\begin{align*}
&\E{\vphantom{\bigg\vert}|R_2|}\leq 
\frac{|n-p-1|}n\sum_{k=1}^{p}\E{\left|\frac{\psi_k}{l_k}\right|}
+\frac2n\sum_{k=1}^{p}\E{\left|\frac{\partial\psi_k}{\partial l_k}\right|}
\\&\qquad
+\frac1n\sum_{k\neq b=1}^p
\left(\E{\left|\psi_k\right|^{2.25}}^\frac1{2.25}+\E{\left|\psi_b\right|^{2.25}}^\frac1{2.25}\right)
\E{\frac1{|l_k-l_b|^{1.8}}}^\frac1{1.8},
\end{align*}}
Moreover, for any $\epsilon>0$, 
{\setlength{\mathindent}{0pt}\begin{align*}
&\E{\sum_{k=1}^p\left|\frac{\psi^*_k}{l_k}\right|}
\leq
\frac{|n-p-1|}n\sum_{k=1}^p\E{\left|\frac{\psi_k}{l_k}\right|^{2(1+\epsilon)}}^{\frac1{1+\epsilon}}
+\frac4n\sum_{k=1}^p\E{\left|\frac{\psi_k}{l_k}\right|^{2(1+\epsilon)}}
\\&\qquad
\cdot\E{\left|\frac{\partial \psi_k}{\partial l_k}\right|^{2(1+\epsilon}}^{\frac1{2(1+\epsilon)}}
+\frac2n\sum_{k\neq b}^p\E{\frac1{|l_k-l_b|^{1.8(1+\epsilon)}}}^\frac1{1.8(1+\epsilon)}
\\&\qquad
\cdot\E{\left|\frac{\psi_k}{l_k}\right|^{4.5(1+\epsilon)}}^\frac1{4.5(1+\epsilon)}
\!\!\left(\E{\left|\psi_k\right|^{4.5(1+\epsilon)}}^\frac1{4.5(1+\epsilon)}+\E{\left|\psi_b\right|^{4.5(1+\epsilon)}}^\frac1{4.5(1+\epsilon)}\right)\!\!.
\end{align*}}
Note that for any $1\leq k\leq p$ and $\epsilon>0$,
\begin{align*}
\E{\left|\psi_k\right|^{4.5}}^\frac1{4.5}\leq\E{\left|\frac{\psi_k}{l_k}\right|^{9(1+\epsilon)}}^\frac1{9(1+\epsilon)}\E{\left|l_k\right|^\frac{9\epsilon}{1+\epsilon}}^\frac{1+\epsilon}{9\epsilon},
\end{align*}
and for any $m>1$, $\E{\left|l_k\right|^m}^2\leq \E{\tr{S^{2m}}}\leq \tr{\Sigma}^{2m}\E{\tr{S\Sigma^{-1}}^{2m}}
=\tr{\Sigma}^{2m}\E{(\chi^2_{np})^{2m}}<\infty$.
Now consider that, for any $m>1$,
\begin{align*}&
\left|\psi_k\right|^m
\leq 
2^{m-1}(\left|\hat\gamma_k\right|^m+\left|\hat\sigma^2\right|^m)
\\&
\left|\frac{\partial\psi_k}{\partial l_k}\right|^m
\leq 
2^{m-1}(\left|\frac{\partial\hat\gamma_k}{\partial l_k}\right|^m+\left|\frac{\partial\hat\sigma^2}{\partial l_k}\right|^m)
\\&
\left|\psi_k\frac{\partial^2\psi_k}{\partial l_k^2}\right|^m
=
\left|\hat\gamma_k+\hat\sigma^2\right|\left|\frac{\partial^2\hat\gamma_k}{\partial l_k^2}+\frac{\partial^2\hat\sigma^2}{\partial l_k^2}\right|^m.
\end{align*}
Therefore, since $\hat\Gamma$ satisfies the weak regularity conditions and $\hat\Sigma\in V_p(\hat\Gamma)$, we obtain
 $\E{|R_1|}<\infty$, $\E{|R_2|}<\infty$ and $\E{\sum_{k=1}^p\left|\frac{\psi^*_k}{l_k}\right|}<\infty$.

Therefore, all the regularity conditions of Lemmas \ref{lem:KONN} and \ref{lem:KONN2} are satisfied, and we have for $\Psi=\text{diag}(\psi_1,...,\psi_p)$
{\setlength{\mathindent}{0pt}\begin{align*}
&\E{L\left(\hat\Sigma,\Sigma\right)}
=
\frac1p\E{\tr{\left[\Sigma^{-1}O\Psi O'\right]^2}-2\tr{\Sigma^{-1}O\Psi O'}+p}
\\&\quad=
\frac1p\E{R_1-2R_2+p}
\\&\quad=
\E{
\frac{(n-p-1)(n-p-2)}{n^2p}\sum_{k=1}^{p}\frac{\psi^{2}_k}{l_k^2}
-\frac{(n-p-1)}{n^2p}\left(\sum_{k=1}^p\frac{\psi_k}{l_k}\right)^2
\right.\\&\qquad\left.
+\frac8{n^2p}\sum_{k=1}^{p}\left(\frac{\partial\psi_k}{\partial l_k}\right)^2
+\frac8{n^2p}\sum_{k=1}^{p}\psi_k\frac{\partial^2\psi_k}{\partial l_k^2}
+8\frac{n-p-1}{n^2p}\sum_{k=1}^{p}\frac{\psi_k}{l_k}\frac{\partial\psi_k}{\partial l_k}
\right.\\&\qquad\left.
+4\frac{n-p-1}{n^2p}\sum_{k\neq b=1}^p\frac{\psi_k}{l_k}\frac{\psi_k-\psi_b}{l_k-l_b}
+\frac8{n^2p}\sum_{k\neq b=1}^p\frac{\partial\psi^n_k}{\partial l_k}\frac{\psi_k-\psi_b}{l_k-l_b}
\right.\\&\qquad\left.
+\frac4{n^2p}\sum_{k\neq b=1}^p\psi^n_k\frac{\frac{\partial\psi_k}{\partial l_k}-\frac{\partial\psi_b}{\partial l_b}}{l_k-l_b}
+\frac4{n^2p}\sum_{k\neq b=1}^p\psi^n_k\frac{\frac{\partial\psi_k}{\partial l_k}-\frac{\partial\psi_b}{\partial l_k}}{l_k-l_b}
\right.\\&\qquad\left.
+\frac2{n^2p}\sum_{k\neq b\neq e=1}^p\frac{\psi_k}{l_k-l_b}\left(\frac{\psi_k-\psi_e}{l_k-l_e}-\frac{\psi_b-\psi_e}{l_b-l_e}\right)
\right.\\&\qquad\left.
+\frac2{n^2p}\sum_{k\neq b\neq e=1}^p\frac{\psi_k-\psi_b}{l_k-l_b}\frac{\psi_k-\psi_e}{l_k-l_e}
-2\frac{n-p-1}{np}\sum_{k=1}^{p}\frac{\psi_k}{l_k}
\right.\\&\qquad\left.
-\frac4{np}\sum_{k=1}^{p}\frac{\partial\psi_k}{\partial l_k}
-\frac2{np}\sum_{k\neq b=1}^p\frac{\psi_k-\psi_b}{l_k-l_b}
+1}.
\end{align*}}
We can now collect the terms of order $1$ and $1/p$, defining
\begin{align*}
&F(\hat\Sigma)
=\frac{(n-p-1)(n-p-2)}{n^2p}\sum_{k=1}^{p}\frac{\psi_k^2}{l_k^2}
-\frac{n-p-1}{n^2p}\left(\sum_{k=1}^p\frac{\psi_k}{l_k}\right)^2
\\&\quad
+4\frac{n-p-1}{n^2p}\sum_{k\neq b=1}^p\frac{\psi_k}{l_k}\frac{\psi_k-\psi_b}{l_k-l_b}
+\frac2{n^2p}\sum_{k\neq b\neq e=1}^p\frac{\psi_k-\psi_b}{l_k-l_b}\frac{\psi_k-\psi_e}{l_k-l_e}
\\&\quad
+\frac2{n^2p}\sum_{k\neq b\neq e=1}^p\frac{\psi_k}{l_k-l_b}\left(\frac{\psi_k-\psi_e}{l_k-l_e}-\frac{\psi_b-\psi_e}{l_b-l_e}\right)
\\&\quad
-2\frac{n-p-1}{np}\sum_{k=1}^{p}\frac{\psi_k}{l_k}
-\frac2{np}\sum_{k\neq b=1}^p\frac{\psi_k-\psi_b}{l_k-l_b}
+1
\end{align*}
and
\begin{align*}
&G(\hat\Sigma)
=
\frac8{n^2p}\sum_{k=1}^{p}\left(\frac{\partial\psi_k}{\partial l_k}\right)^2
+\frac8{n^2p}\sum_{k=1}^{p}\psi_k\frac{\partial^2\psi_k}{\partial l_k^2}
+8\frac{n-p-1}{n^2p}\sum_{k=1}^{p}\frac{\psi_k}{l_k}\frac{\partial\psi_k}{\partial l_k}
\\&\quad
+\frac8{n^2p}\sum_{k\neq b=1}^p\frac{\partial\psi_k}{\partial l_k}\frac{\psi_k-\psi_b}{l_k-l_b}
+\frac4{n^2p}\sum_{k\neq b=1}^p\psi_k\frac{\frac{\partial\psi_k}{\partial l_k}-\frac{\partial\psi_b}{\partial l_b}}{l_k-l_b}
\\&\quad
+\frac4{n^2p}\sum_{k\neq b=1}^p\psi_k\frac{\frac{\partial\psi_k}{\partial l_k}-\frac{\partial\psi_b}{\partial l_k}}{l_k-l_b}
-\frac4{np}\sum_{k=1}^{p}\frac{\partial\psi_k}{\partial l_k}
\end{align*}
so that $\E{L(\hat\Sigma,\Sigma)}=\E{F(\hat\Sigma)+G(\hat\Sigma)}$, with $\E{\left|F(\hat\Sigma)+G(\hat\Sigma)\right|}\leq\E{|R_1|}+2\E{|R_2|}+p<\infty$, as desired. Plugging in $\psi_k=\hat\gamma_k+\hat\sigma^2$ yields, after a bit of algebra:
{\setlength{\mathindent}{5pt}\begin{align}
&F(\hat\Gamma+\hat\sigma^2I)
=
\frac{(n-p-1)(n-p-2)}{n^2p}\sum_{k=1}^{\rho}\frac{\hat\gamma_k^2}{l_k^2}
\notag\\&\quad
+2\frac{(n-p-1)(n-p-2)}{n^2p}\sum_{k=1}^{\rho}\frac{\hat\gamma_k\hat\sigma^2}{l_k^2}
+\frac{(n-p-1)(n-p-2)}{n^2p}\sum_{c=1}^{p}\frac{\hat\sigma^4_\rho}{l_c^2}
\notag\\&\quad
-\frac{n-p-1}{n^2p}\left(\sum_{k=1}^\rho\frac{\hat\gamma_k}{l_k}\right)^2
-2\frac{n-p-1}{n^2p}\sum_{c=1}^p\frac{\hat\sigma^2}{l_c}\sum_{k=1}^\rho\frac{\hat\gamma_k}{l_k}
\notag\\&\quad
-\frac{n-p-1}{n^2p}\left(\sum_{c=1}^p\frac{\hat\sigma^2}{l_c}\right)^2
+4\frac{n-p-1}{n^2p}\sum_{k\neq b}^\rho\frac{\hat\gamma_k}{l_k}\frac{\hat\gamma_k-\hat\gamma_b}{l_k-l_b}
\notag\\&\quad
+4\frac{n-p-1}{n^2p}\sum_{k=1}^\rho\sum_{c=\rho+1}^p\frac{\hat\gamma_k}{l_k}\frac{\hat\gamma_k}{l_k-l_c}
+4\frac{n-p-1}{n^2p}\sum_{k=1}^\rho\sum_{c=\rho+1}^p\frac{\hat\sigma^2}{l_c}\frac{\hat\gamma_k}{l_k-l_c}
\notag\\&\quad
+\frac2{n^2p}\sum_{k\neq b\neq e=1}^\rho\frac{\hat\gamma_k-\hat\gamma_b}{l_k-l_b}\frac{\hat\gamma_k-\hat\gamma_e}{l_k-l_e}
+\frac2{n^2p}\sum_{k\neq b=1}^\rho\sum_{c=\rho+1}^p\frac{\hat\gamma_k}{l_k-l_c}\frac{\hat\gamma_b}{l_b-l_c}
\notag\\&\quad
-\frac6{n^2p}\sum_{k\neq b}^\rho\sum_{c=\rho+1}^p\frac{\hat\gamma_k+\hat\sigma^2}{l_k-l_c}\frac{\hat\gamma_b}{l_b-l_c}
+\frac6{n^2p}\sum_{k=1}^\rho\sum_{c\neq d}^p\frac{\hat\gamma_k+\hat\sigma^2}{l_k-l_c}\frac{\hat\gamma_k}{l_k-l_d}
\notag\\&\quad
-\frac2{n^2p}\sum_{k=1}^\rho\sum_{c\neq d}^p\frac{\hat\gamma_k}{l_k-l_c}\frac{\hat\gamma_k}{l_k-l_d}
+\frac2{n^2p}\sum_{k\neq b\neq e=1}^\rho\frac{\hat\gamma_k}{l_k-l_b}\left(\frac{\hat\gamma_k-\hat\gamma_e}{l_k-l_e}-\frac{\hat\gamma_b-\hat\gamma_e}{l_b-l_e}\right)
\notag\\&\quad
+\frac4{n^2p}\sum_{k\neq b=1}^p\sum_{c=\rho+1}^p\frac{\hat\gamma_k-\hat\gamma_b}{l_k-l_b}\frac{\hat\gamma_k}{l_k-l_c}
+\frac6{n^2p}\sum_{k\neq b}^\rho\sum_{c=\rho+1}^p\frac{\hat\gamma_k-\hat\gamma_b}{l_k-l_b}\frac{\hat\gamma_k+\hat\sigma^2}{l_k-l_c}
\notag\\&\quad
-2\frac{n-p-1}{np}\sum_{k=1}^{\rho}\frac{\hat\gamma_k}{l_k}
-2\frac{n-p-1}{np}\sum_{c=1}^{p}\frac{\hat\sigma^2}{l_c}
-\frac2{np}\sum_{k\neq b=1}^\rho\frac{\hat\gamma_k-\hat\gamma_k}{l_k-l_b}
\notag\\&\quad
-\frac4{np}\sum_{k=1}^\rho\sum_{c=\rho+1}^p\frac{\hat\gamma_k}{l_k-l_c}
+1
\label{eq:CF-ureF}
\end{align}}
and
{\setlength{\mathindent}{5pt}\begin{align}
&G(\hat\Gamma+\hat\sigma^2I)
=
\frac8{n^2p}\sum_{k=1}^{\rho}\left(\frac{\partial\hat\gamma_k}{\partial l_k}\right)^2
+\frac{16}{n^2p}\sum_{k=1}^{\rho}\frac{\partial\hat\gamma_k}{\partial l_k}\frac{\partial\hat\sigma^2}{\partial l_k}
+\frac8{n^2p}\sum_{k=1}^{p}\left(\frac{\partial\hat\sigma^2}{\partial l_k}\right)^2
\notag\\&\quad
+\frac8{n^2p}\sum_{k=1}^\rho\hat\gamma_k\frac{\partial^2\hat\gamma_k}{\partial l_k^2}
+\frac8{n^2p}\sum_{k=1}^{p}\hat\sigma^2\frac{\partial^2\hat\sigma^2}{\partial l_k^2}
+8\frac{n-p-1}{n^2p}\sum_{k=1}^{\rho}\frac{\hat\gamma_k}{l_k}\frac{\partial\hat\gamma_k}{\partial l_k}
\notag\\&\quad
+8\frac{n-p-1}{n^2p}\sum_{k=1}^{\rho}\frac{\hat\gamma_k}{l_k}\frac{\partial\hat\sigma^2}{\partial l_k}
+8\frac{n-p-1}{n^2p}\sum_{k=1}^{\rho}\frac{\hat\sigma^2}{l_k}\frac{\partial\hat\gamma_k}{\partial l_k}
\notag\\&\quad
+8\frac{n-p-1}{n^2p}\sum_{k=1}^{p}\frac{\hat\sigma^2}{l_k}\frac{\partial\hat\sigma^2}{\partial l_k}
+\frac8{n^2p}\sum_{k\neq b=1}^\rho\frac{\partial\hat\gamma_k}{\partial l_k}\frac{\hat\gamma_k-\hat\gamma_k}{l_k-l_b}
\notag\\&\quad
+\frac8{n^2p}\sum_{k=1}^\rho\sum_{c=\rho+1}^p\frac{\partial\hat\gamma_k}{\partial l_k}\frac{\hat\gamma_k}{l_k-l_c}
+\frac8{n^2p}\sum_{k=1}^\rho\sum_{c=\rho+1}^p\frac{\partial\hat\sigma^2}{\partial l_c}\frac{\hat\gamma_k}{l_k-l_c}
\notag\\&\quad
+\frac4{n^2p}\sum_{k\neq b=1}^\rho\hat\gamma_k\frac{\frac{\partial\hat\gamma_k}{\partial l_k}-\frac{\partial\hat\gamma_b}{\partial l_b}}{l_k-l_b}
+\frac4{n^2p}\sum_{k=1}^\rho\sum_{c=\rho+1}^p\hat\gamma_k\frac{\frac{\partial\hat\gamma_k}{\partial l_k}-\frac{\partial\hat\sigma^2}{\partial l_c}}{l_k-l_c}
\notag\\&\quad
+\frac4{n^2p}\sum_{k=1}^\rho\sum_{c=\rho+1}^p\hat\sigma^2\frac{\frac{\partial\hat\gamma_k}{\partial l_k}-\frac{\partial\hat\sigma^2}{\partial l_c}}{l_k-l_c}
+\frac4{n^2p}\sum_{c\neq d=\rho+1}^p\hat\sigma^2\frac{\frac{\partial\hat\sigma^2}{\partial l_c}-\frac{\partial\hat\sigma^2}{\partial l_d}}{l_c-l_d}
\notag\\&\quad
+\frac4{n^2p}\sum_{k\neq b=1}^\rho\hat\gamma_k\frac{\frac{\partial\hat\gamma_k}{\partial l_k}-\frac{\partial\hat\gamma_b}{\partial l_k}}{l_k-l_b}
+\frac4{n^2p}\sum_{k=1}^\rho\sum_{c=\rho+1}^p\hat\gamma_k\frac{\frac{\partial\hat\gamma_k}{\partial l_k}-\frac{\partial\hat\sigma^2}{\partial l_k}}{l_k-l_c}
\notag\\&\quad
+\frac4{n^2p}\sum_{k=1}^\rho\sum_{c=\rho+1}^p\hat\sigma^2\frac{\frac{\partial\hat\gamma_k}{\partial l_c}-\frac{\partial\hat\sigma^2}{\partial l_c}}{l_k-l_c}
+\frac4{n^2p}\sum_{c\neq d=\rho+1}^p\hat\sigma^2\frac{\frac{\partial\hat\sigma^2}{\partial l_c}-\frac{\partial\hat\sigma^2}{\partial l_c}}{l_c-l_d}
\notag\\&\quad
-\frac4{np}\sum_{k=1}^\rho\frac{\partial\hat\gamma_k}{\partial l_k}
-\frac4{np}\sum_{k=1}^{p}\frac{\partial\hat\sigma^2}{\partial l_k}.
\label{eq:CF-ureG}
\end{align}}

For the second part of the theorem, we see that
{\setlength{\mathindent}{0pt}\begin{align*}
&\E{\left|F(\hat\Sigma)\right|}
\leq
1+
\frac{|n-p-1||n-p-2|}{n^2}\E{\sup_{p\in\mathbb{N}^*}\max_{1\leq k\leq p}\left|\frac{\psi_k}{l_k}\right|^2}
\\&\quad
+\frac{|n-p-1|p}{n^2}\E{\sup_{p\in\mathbb{N}^*}\max_{1\leq k\leq p}\left|\frac{\psi_k}{l_k}\right|^2}
\\&\quad
+4\frac{|n-p-1|(p-1)}{n^2}\E{\sup_{p\in\mathbb{N}^*}\max_{1\leq k\leq p}\left|\frac{\psi_k}{l_k}\right|^2}^{\frac12}\E{\sup_{p\in\mathbb{N}^*}\max_{1\leq k\neq b\leq p}\left|\frac{\psi_k-\psi_b}{l_k-l_b}\right|^2}^{\frac12}
\\&\quad
+\frac{2(p-1)(p-2)}{n^2}\E{\sup_{p\in\mathbb{N}^*}\max_{1\leq k\neq b\leq p}\left|\frac{\psi_k-\psi_b}{l_k-l_b}\right|^2}
\\&\quad
+\frac{2(p-1)(p-2)}{n^2}\E{\sup_{p\in\mathbb{N}^*}\max_{1\leq k\leq p}\left|\frac{\psi_k}{l_k}\right|^2}^{\frac12}
\\&\qquad\qquad
\cdot\E{\sup_{p\in\mathbb{N}^*}\max_{1\leq k\neq b\neq e\leq p}\left|\frac{l_k}{l_k-l_b}\left(\frac{\psi_k-\psi_e}{l_k-l_e}-\frac{\psi_b-\psi_e}{l_b-l_e}\right)\right|^2}^{\frac12}
\\&\quad
-2\frac{|n-p-1|}{n}\E{\sup_{p\in\mathbb{N}^*}\max_{1\leq k\leq p}\left|\frac{\psi_k}{l_k}\right|}
-\frac{2(p-1)}{n}\E{\sup_{p\in\mathbb{N}^*}\max_{1\leq k\neq b\leq p}\left|\frac{\psi_k-\psi_b}{l_k-l_b}\right|}
\end{align*}}
and
{\setlength{\mathindent}{0pt}\begin{align*}
&p\E{\left|G(\hat\Sigma)\right|}
\leq
\frac{8p}{n^2}\E{\sup_{p\in\mathbb{N}^*}\max_{1\leq k\leq p}\left|\frac{\partial\psi_k}{\partial l_k}\right|^2}
+\frac{8p}{n^2}\E{\sup_{p\in\mathbb{N}^*}\max_{1\leq k\leq p}\left|\psi_k\frac{\partial^2\psi_k}{\partial l_k^2}\right|}
\\&\quad
+8\frac{|n-p-1|p}{n}\E{\sup_{p\in\mathbb{N}^*}\max_{1\leq k\leq p}\left|\frac{\psi_k}{l_k}\right|^2}^{\frac12}\E{\sup_{p\in\mathbb{N}^*}\max_{1\leq k\leq p}\left|\frac{\partial\psi_k}{\partial l_k}\right|^2}^\frac12
\\&\quad
+\frac{8(p-1)p}{n^2}\E{\sup_{p\in\mathbb{N}^*}\max_{1\leq k\leq p}\left|\frac{\partial\psi_k}{\partial l_k}\right|^2}^\frac12\E{\sup_{p\in\mathbb{N}^*}\max_{1\leq k\neq b\leq p}\left|\frac{\psi_k-\psi_b}{l_k-l_b}\right|^2}^{\frac12}
\\&\quad
+\frac{4(p-1)p}{n^2}\E{\sup_{p\in\mathbb{N}^*}\max_{1\leq k\leq p}\left|\frac{\psi_k}{l_k}\right|^2}^{\frac12}\!\!\!
\E{\sup_{p\in\mathbb{N}^*}\max_{\substack{1\leq k\\ \neq b\leq p}}\left|\frac{l_k}{l_k-l_b}\left(\frac{\partial\psi_k}{\partial l_k}-\frac{\partial\psi_b}{\partial l_b}\right)\right|^2}^\frac12
\\&\quad
+\frac{4(p-1)p}{n^2}\E{\sup_{p\in\mathbb{N}^*}\max_{1\leq k\leq p}\left|\frac{\psi_k}{l_k}\right|^2}^{\frac12}\!\!\!
\E{\sup_{p\in\mathbb{N}^*}\max_{\substack{1\leq k\\ \neq b\leq p}}\left|\frac{l_k}{l_k-l_b}\left(\frac{\partial\psi_k}{\partial l_k}-\frac{\partial\psi_b}{\partial l_k}\right)\right|^2}^\frac12
\\&\quad
-\frac{4p}{n}\E{\sup_{p\in\mathbb{N}^*}\max_{1\leq k\leq p}\left|\frac{\partial\psi_k}{\partial l_k}\right|}.
\end{align*}}
Again, one can proceed like in the weak case to see that if $\hat\Gamma$ satisfies its strong regularity conditions and $\hat\Sigma\in\tilde V_p(\hat\Gamma)$, we get $\lim\limits_{n\rightarrow\infty}\E{\big|F(\hat\Sigma)\big|}<\infty$ and $\lim\limits_{n\rightarrow\infty}p\E{\big|G(\hat\Sigma)\big|}<\infty$ as $\lim\limits_{n\rightarrow\infty}\frac{p_n}n\in(0,1)$, as desired.
\end{proof}

\begin{proof}[\bf Proof of Proposition \ref{prop:MINM}]
First note any element of $C_c^\infty(H_p;\mathbb{R})$, the space of smooth, compactly supported functions from $H_p$ to $\mathbb{R}$, satisfies the weak regularity conditions of Definition \ref{defn:WRC} for any weak $\hat\Gamma$.
Now, if $\tilde\Sigma$ is a minimum over $V_p(\hat\Gamma)$, then for any $\eta\in C_c^\infty(H^p_+;\mathbb{R})$ and any $t\in\mathbb{R}$, $\tilde\sigma^2+t\eta$ satisfies the $\hat\Gamma$-weak regularity conditions too and  $\epsilon\rightarrow\E{F(\hat\Gamma+[\tilde\sigma^2+t\eta]I)}$ is smooth over $\mathbb{R}$ with a minimum at $t=0$.
But we find that the first variation satisfies
{\setlength{\mathindent}{5pt}\begin{align}
&\frac{\partial}{\partial t}\E{F(\hat\Gamma+[\tilde\sigma^2+t\eta]I)}\bigg\vert_{t=0}
=
\E{\eta\cdot\left(
2\frac{(n-p-1)(n-p-2)}{n^2p}\sum_{k=1}^{{\hat\rho}}\frac{\hat\gamma_k}{l_k^2}
\right.\right.\notag\\&\quad\left.\left.
+2\frac{(n-p-1)(n-p-2)}{n^2p}\sum_{c=1}^{p}\frac{\tilde\sigma^2}{l_c^2}
-2\frac{n-p-1}{n^2p}\sum_{c=1}^p\frac{1}{l_c}\sum_{k=1}^{\hat\rho}\frac{\hat\gamma_k}{l_k}
\right.\right.\notag\\&\quad\left.\left.
-2\frac{n-p-1}{n^2p}\sum_{c=1}^p\frac{1}{l_c}\sum_{c=1}^p\frac{\tilde\sigma^2}{l_c}
+4\frac{n-p-1}{n^2p}\sum_{k=1}^{\hat\rho}\sum_{c={\hat\rho}+1}^p\frac{1}{l_c}\frac{\hat\gamma_k}{l_k-l_c}
\right.\right.\notag\\&\quad\left.\left.
-\frac6{n^2p}\sum_{k\neq b}^{\hat\rho}\sum_{c={\hat\rho}+1}^p\frac{1}{l_k-l_c}\frac{\hat\gamma_b}{l_b-l_c}
+\frac6{n^2p}\sum_{k=1}^{\hat\rho}\sum_{c\neq d={\hat\rho}+1}^p\frac{1}{l_k-l_c}\frac{\hat\gamma_k}{l_k-l_d}
\right.\right.\notag\\&\quad\left.\left.
+\frac6{n^2p}\sum_{k\neq b}^{\hat\rho}\sum_{c={\hat\rho}+1}^p\frac{\hat\gamma_k-\hat\gamma_b}{l_k-l_b}\frac{1}{l_k-l_c}
-2\frac{n-p-1}{np}\sum_{c=1}^{p}\frac{1}{l_k}
\right)}
\notag\\&\qquad=
\E{\eta \cdot F_1\big[l,\hat\rho,\hat\gamma,\tilde\sigma^2\big]}
\notag\\&\qquad=
\int_{H_p}\eta(l_1,...,l_p)F_1\big[l,\hat\rho,\hat\gamma,\tilde\sigma^2\big] \cdot f_{l_1,...,l_p}(l_1,...,l_p)\prod_{i=1}^pdl_i,
\label{eq:MINM-F1def}
\end{align}}
where $f_{l_1,...,l_p}(l_1,...,l_p)$ stands for the p.d.f. of $l_1>...>l_p$. Now, if this equals zero for all $\eta\in C_c^\infty(H^p_+;\mathbb{R})$, by the fundamental lemma of calculus of variations (see, say, \cite{GiaquintaHildebrandt96} ch.\ 2.2) we obtain $F_1\big[l,\hat\rho,\hat\gamma,\tilde\sigma^2\big] \cdot f_{l_1,...,l_p}(l_1,...,l_p)\equiv 0$, that is, $F_1\big[l,\hat\rho,\hat\gamma,\tilde\sigma^2\big]\equiv 0$. This implies $\tilde\sigma^2=A/B$.

For the second statement, notice that by construction, the space of $\hat\Gamma$-weak noise estimators is convex; let $\hat\sigma^2$ be some arbitrary element. Define $H: [0,1]\rightarrow \mathbb{R}$ to be the smooth function
\[
H(t)=\E{F(\hat\Gamma+[\tilde\sigma^2+t(\hat\sigma^2-\tilde\sigma^2)]I)}.
\]
Notice that, for  $F_1$ as in eq.\ (\ref{eq:MINM-F1def}),
\begin{align*}
H'(0)=\E{(\hat\sigma^2-\tilde\sigma^2)\cdot F_1\big[l,\hat\rho,\hat\gamma,\tilde\sigma^2\big]}=0,
\end{align*}
since $\tilde\sigma^2=A/B$.
Moreover,
{\setlength{\mathindent}{5pt}\begin{align}
&H''(t)=\frac{\partial^2}{\partial t^2}\E{F(\hat\Gamma_r+[\tilde\sigma^2+t(\hat\sigma^2-\tilde\sigma^2)]I)}
\notag\\&\quad=\;
2\frac{n-p-1}{n^2p}\E{\left(\hat\sigma^2-\tilde\sigma^2\right)^2\left(
(n-p-2)\sum_{c=1}^{p}\frac{1}{l_c^2}
-\left(\sum_{c=1}^p\frac{1}{l_c}\right)^2
\right)}
\label{eq:MINM-secder}\\&\quad\geq\;
2\frac{(n-p-1)(n-2p-2)}{n^2p^2}\E{\left(\hat\sigma^2-\tilde\sigma^2\right)^2\left(
\sum_{c=1}^p\frac{1}{l_c}\right)^2}
\quad\left(\parbox{1.65cm}{Jensen's\\inequality}\right)
\notag\\&\quad\geq\;
0,
\notag
\end{align}}
for $n\geq 2p+2$. Therefore, by integration by parts
\begin{align*}
&\E{F(\hat\Gamma+\hat\sigma^{2}I)}-\E{F(\hat\Gamma+\tilde\sigma^2I)}
\\&\qquad=
H(1)-H(0)=\int_0^1 (1-t)H''(t)\,dt
\; \geq0.
\end{align*}
Since this is true for any $\hat\Gamma$-weak noise estimator $\hat\sigma^2$, we conclude that $\tilde\Sigma$ is a minimum over $V_p(\hat\Gamma)$, as desired.
\end{proof}

\subsection{Proofs for Section \ref{sec:P}}

\begin{proof}[\bf Proof of Lemma \ref{lem:LLN}]
To simplify notation in what follows, define $c_\pm=[1\pm\sqrt{c}]^2$. In the proof of Theorem 2.3 in \cite{Nadler08}, p.\ 2807, it is remarked that for $\sigma^2=1$ and $\rho=1$, the empirical distribution of $l_2,...,l_p$ converges a.s. to a Mar\v{c}enko-Pastur distribution with parameter $c$. That is, for the truncated empirical spectral measure
$d\mu_p=\frac1{p-\rho}\sum_{c=\rho+1}^pd\delta_{l_i}$
(where the $\delta$ are Dirac measures) we have weak convergence $d\mu_p \Rightarrow d\mu_\text{MP(c)}$ a.s. where
\begin{align*}
d\mu_\text{MP(c)}= \frac{\sqrt{(c_+-t)(t-c_-)}}{2\pi ct}\1{\vphantom{\bigg\vert}c_-\leq t\leq c_+}dt.
\end{align*}
As noted by the author, the argument carries on for $\rho\neq 1$, and if $\sigma\neq 1$ we can apply the argument to $l_{\rho+1}/\sigma^2,...,l_p/\sigma^2$ to obtain $d\mu_p \Rightarrow d\mu_{\sigma^2\text{MP}(c)}$ a.s., where
\begin{align*}
d\mu_{\sigma^2\text{MP}(c)}=& \frac{\sqrt{(\sigma^2c_+-t)(t-\sigma^2c_-)}}{2\pi c\sigma^2t}\1{\vphantom{\bigg\vert}\sigma^2c_-\leq t\leq \sigma^2c_+}dt.
\end{align*}
\textit{Part (i)} Applying the results of \cite{BaikSilverstein06}, Theorem 1.1 to $l_k/\sigma^2$ and $l_{\rho+1}/\sigma^2$ we obtain:
\begin{align*}
&l_k\xrightarrow[n\rightarrow\infty]{\text{a.s.}}\frac{(\gamma_k+\sigma^2)(\gamma_k+c\sigma^2)}{\gamma_k},
\qquad\qquad
l_{\rho+1}\xrightarrow[n\rightarrow\infty]{\text{a.s.}}c_+\sigma^2.
\end{align*}
We will write $\bar l_k=(\gamma_k+\sigma^2)(\gamma_k+c\sigma^2)/\gamma_k$ to simplify notation. Let the underlying sample space be denoted $\Omega$. Since $\gamma_\rho>\sqrt{c}\sigma^2$, we have $\bar l_k-c_+\sigma^2=M$ for some $M>0$. Therefore, for almost all $\omega\in\Omega$, there exists an $N_1(\omega)$ such that $\forall n>N(\omega)$, $l^p_k(\omega)-l^p_p(\omega)>...>l^p_k(\omega)-l^p_{\rho+1}(\omega)>M/2$ and $l^p_p(\omega)<...<l^p_{\rho+1}(\omega)<c_+\sigma^2+M/2$. Moreover, for any $\epsilon>0$, there must be an $N_2(\omega)$ such that for all $n>N_2(\omega,\epsilon)$, $|\bar l_k-l^p_k(\omega)|<\epsilon$. Notice that we can write, for any $n>N_1(\omega)\vee N_2(\omega,\epsilon)$,
\begin{align*}
\frac1{p-\rho}\sum_{c=\rho+1}^p\frac{l^p_c(\omega)}{l^p_k(\omega)-l^p_c(\omega)}
=&\frac1{p-\rho}\sum_{c=\rho+1}^p\frac{[\bar l_k-l^p_k(\omega)]l^p_c(\omega)}{[l^p_k(\omega)-l^p_c(\omega)][\bar l_k-l^p_c(\omega)]}
\\&
+\int_{(0,c_+\sigma^2+\frac{M}2)}\frac{t}{\bar l_k-t}d\mu_p(t,\omega),
\end{align*}
and
\begin{align}
\left|\frac1{p-\rho}\sum_{c=\rho+1}^p\frac{[\bar l_k-l^p_k(\omega)]l^p_c(\omega)}{[l^p_k(\omega)-l^p_c(\omega)][\bar l_k-l^p_c(\omega)]}\right|<\left[\frac{4}{M^2}c_+\sigma^2+\frac 2M\right]\epsilon.\label{eq:lemLLN-bound1-1}
\end{align}
But $0<t/(\bar l_k-t)<1+2c_+\sigma^2/M$ on $t\in(0,c_+\sigma^2+\frac{M}2)$, and it is certainly continuous. Therefore, by the portmanteau theorem of weak convergence of measures,
\begin{align}
&\lim_{n\rightarrow\infty}\int_{(0,c_+\sigma^2+\frac{M}2)}\frac{t}{\bar l_k-t}d\mu_p(t,\omega)=\int_{(0,c_+\sigma^2+\frac{M}2)}\frac{t}{\bar l_k-t}d\mu_{\sigma^2\text{MP}(c)}(t)
\notag\\&\quad
=\int_{c_-\sigma^2}^{c^+\sigma^2}\frac{t}{\bar l_k-t}\frac{\sqrt{(\sigma^2c_+-t)(t-\sigma^2c_-)}}{2\pi c\sigma^2t}dt
\notag\\&\quad
=\frac{2\sigma^2}{\bar l_k-(c+1)\sigma^2+\sqrt{\vphantom{\big\vert}[\bar l_k-(c+1)\sigma^2]^2-4c\sigma^4}}.\label{eq:lemLLN-limit1}
\end{align}
But by definition, $\bar l_k=(\gamma_k+\sigma^2)(\gamma_k+c\sigma^2)/\gamma_k$ which can be rewritten as a quadratic equation in $\gamma_k$,
\begin{align*}
\gamma_k^2-[(\bar l_k-(c+1)\sigma^2]\gamma_k+c\sigma^4=0.
\end{align*}
The roots are
\begin{align*}
\frac12[\bar l_k-(c+1)\sigma^2]\pm\frac12\sqrt{[\bar l_k-(c+1)\sigma^2]^2-4c\sigma^4},
\end{align*}
and notice that twice the negative root satisfies
\begin{align*}
&[\bar l_k-(c+1)\sigma^2]-\sqrt{[\bar l_k-(c+1)\sigma^2]^2-4c\sigma^4}
\\&\qquad=\;
[\bar l_k-(c+1)\sigma^2]
\\&\qquad\qquad\quad
-\sqrt{\left([\bar l_k-(c+1)\sigma^2]-2\sqrt{c}\sigma^2\right)\left([\bar l_k-(c+1)\sigma^2]+2\sqrt{c}\sigma^2\right)}
\\&\qquad\leq\;
[\bar l_k-(c+1)\sigma^2]-[\bar l_k-(c+1)\sigma^2]+2\sqrt{c}\sigma^2
\\&\qquad=\;
2\sqrt{c}\sigma^2.
\end{align*}
Therefore, $\gamma_k$ cannot equal the negative root, because it would imply $\gamma_k\leq \sqrt{c}\sigma^2$, a contradiction. So  $\gamma_k$ equals the positive root, which, plugged in eq.\ (\ref{eq:lemLLN-limit1}), yields
\begin{align*}
&\lim_{n\rightarrow\infty}\int_{(0,c_+\sigma^2+\frac{M}2)}\frac{t}{\bar l_k-t}d\mu_p(t,\omega)=\frac{\sigma^2}{\gamma_k}.
\end{align*}
Hence, for some $N_3(\omega,\epsilon)$, we have for all $n>N_3(\omega,\epsilon)$
\begin{align*}
&\left|\int_{(0,c_+\sigma^2+\frac{M}2)}\frac{t}{\bar l_k-t}d\mu_p(t,\omega)-\frac{\sigma^2}{\gamma_k}\right|<\epsilon.
\end{align*}
Therefore, from eq.\ (\ref{eq:lemLLN-bound1-1}), we obtain that for $n>N_1(\omega)\vee N_2(\omega,\epsilon)\vee N_3(\omega,\epsilon)$
\begin{align*}
&\left|\frac1{p-\rho}\sum_{c=\rho+1}^p\frac{l^p_c(\omega)}{l^p_k(\omega)-l^p_c(\omega)}
-\frac{\sigma^2}{\gamma_k}\right|
<
\left[\frac{4}{M^2}c_+\sigma^2+\frac 2M+1\right]\epsilon.
\end{align*}
Since $\epsilon>0$ is arbitrary, we conclude that for almost all $\omega\in\Omega$,
\begin{align*}
&\lim_{n\rightarrow\infty}\frac1{p-\rho}\sum_{c=\rho+1}^p\frac{l^p_c(\omega)}{l^p_k(\omega)-l^p_c(\omega)}
=\frac{\sigma^2}{\gamma_k},
\end{align*}
as desired.

\textit{Part (ii)} The proof is similar in spirit to the previous one, but simpler. Applying the results of \cite{BaikSilverstein06}, Theorem 1.1 to $l_p/\sigma^2$ we obtain:
\begin{align*}
l_p\xrightarrow[n\rightarrow\infty]{\text{a.s.}}c_-\sigma^2.
\end{align*}
Therefore, for almost all $\omega\in\Omega$, there is a $N(\omega)$ such that $\forall n>N(\omega)$, $l^p_{\rho+1}(\omega)>...>l^p_p(\omega)>c_-\sigma^2/2$. Hence, for $n>N(\omega)$,
\begin{align*}
&\frac1{p-\rho}\sum_{c=\rho+1}^p\frac{1}{l_c^{pm}(\omega)}
=\int_{\left(\frac{c_-\sigma^2}2,\infty\right)}\frac1{t^m}d\mu_p(t,\omega),
\end{align*}
Certainly, $0<1/t^m<\left(\frac2{c_-\sigma^2}\right)^m$ for $t\in \left(\frac{c_-\sigma^2}2,\infty\right)$, and $1/t^m$ is continuous there. Thus, by the portmanteau theorem of weak convergence of measures, we have for almost all $\omega$
\begin{align*}
&\lim_{n\rightarrow\infty}\frac1{p-\rho}\sum_{c=\rho+1}^p\frac{1}{l_c^{pm}(\omega)}
=\int_{\left(\frac{c_-\sigma^2}2,\infty\right)}\frac1{t^m}d\mu_{\sigma^2\text{MP}(c)}(t)
\\&\qquad=
\int_{c_-\sigma^2}^{c^+\sigma^2}\frac1{t^m}\frac{\sqrt{(\sigma^2c_+-t)(t-\sigma^2c_-)}}{2\pi c\sigma^2t}dt
\\&\qquad=
\frac1{(1-c)^{2m-1}}\frac1{\sigma^{2m}},
\end{align*}
which concludes the proof.
\end{proof}

\begin{proof}[\bf Proof of Theorem \ref{thm:NORM}]
Recall the definition of $\tilde\sigma^2$ as $A/B$ from proposition \ref{prop:MINM}.

~\newline
{\it Part (i)} It follows easily from the Lemma \ref{lem:LLN} that
\begin{equation}A\xrightarrow[n\rightarrow\infty]{\text{a.s.}}\frac1{\sigma^2}
\qquad\text{ and }\qquad
B\xrightarrow[n\rightarrow\infty]{\text{a.s.}}\frac1{\sigma^{4}}.\label{eq:NORM-d1}
\end{equation}
The result then follows immediately. 

~\newline
{\it Part (ii)} 
Denote by $\bar\gamma_k$ the a.s. finite limit $\lim\limits_{n\rightarrow\infty}\hat\gamma_k$. Start by writing $A=\frac{n-p-1}{np}\sum\limits_{c=1}^p\frac1{l_c}+E_n$. We find
{\setlength{\mathindent}{5pt}\begin{align}
&n(\tilde\sigma^2-\sigma^2)
\notag\\&\quad=
n\sigma^2\left[\frac{\frac{n-p-1}{np}\sum\limits_{c=1}^p\frac{\sigma^2}{l_c}}{\frac{(n-p-1)(n-p-2)}{n^2p}\sum\limits_{c=1}^{p}\frac{\sigma^4}{l_c^2}
-\frac{(n-p-1)p}{n^2}\left(\frac1p\sum\limits_{c=1}^p\frac{\sigma^2}{l_c}\right)^2}-1\right]+\frac{nE_n}{B}
\notag\\&\quad=
\frac{\sigma^2}{\sigma^4 B}\left[
n\left(\frac{n-p-1}{np}\sum\limits_{c=1}^p\frac{\sigma^2}{l_c}-1\right)
\right.\notag\\&\quad\qquad\qquad\left.
+n\left(\frac1{1-\frac{p-\hat\rho}n}-\frac{(n-p-1)(n-p-2)}{n^2p}\sum\limits_{c=1}^{p}\frac{\sigma^4}{l_c^2}\right)
\right.\notag\\&\quad\qquad\qquad\left.
+n\left(\frac{(n-p-1)p}{n^2}\left(\frac1p\sum\limits_{c=1}^p\frac{\sigma^2}{l_c}\right)^2-\frac{\frac{p-\hat\rho}n}{1-\frac{p-\hat\rho}n}\right)
\right]+\frac{nE_n}{B}.
\label{eq:NORM-a1}
\end{align}}
First consider $nE_n$. We know the asymptotic behavior of all terms except $\frac3{np}\sum\limits_{k=1}^{\hat\rho}\sum\limits_{c\neq d=r+1}^p \frac1{l_k-l_c}\frac{\hat\gamma_k}{l_k-l_d}$, which we can crudely bound as
$$
0<
\frac3{np}\sum_{k=1}^{\hat\rho}\sum_{c\neq d=\hat\rho+1}^p \frac1{l_k-l_c}\frac{\hat\gamma_k}{l_k-l_d}
<\frac3{np}\sum_{k=1}^{\hat\rho}\hat \gamma_k\left(\sum_{c=\hat\rho+1}^p \frac1{l_k-l_c}\right)^2.
$$
Therefore, we obtain that 
\begin{align}&
\frac{(1-c)^2}c\sum_{k=1}^{\hat\rho}\frac{\gamma_k^2\bar\gamma_k}{(\gamma_k+\sigma^2)^2(\gamma_k+c\sigma^2)^2}
+\frac1{\sigma^2}\sum_{k=1}^{\hat\rho}\frac{\gamma_k\bar\gamma_k}{(\gamma_k+\sigma^2)(\gamma_k+c\sigma^2)}
\notag\\&\qquad\qquad
-2\frac1{\sigma^2}\sum_{k=1}^{\hat\rho}\frac{\gamma_k\bar\gamma_k}{(\gamma_k+c\sigma^2)^2}
\notag\\&\hspace{100pt}\geq\quad
\lim_{n\rightarrow\infty}nE_n
\quad\geq
\notag\\&
\frac{(1-c)^2}c\sum_{k=1}^{\hat\rho}\frac{\gamma_k^2\bar\gamma_k}{(\gamma_k+\sigma^2)^2(\gamma_k+c\sigma^2)^2}
+\frac1{\sigma^2}\sum_{k=1}^{\hat\rho}\frac{\gamma_k\bar\gamma_k}{(\gamma_k+\sigma^2)(\gamma_k+c\sigma^2)}
\notag\\&\qquad\qquad
-2\frac1{\sigma^2}\sum_{k=1}^{\hat\rho}\frac{\gamma_k\bar\gamma_k}{(\gamma_k+c\sigma^2)^2}
-3c\sum_{k=1}^{\hat\rho}\frac{\bar\gamma_k}{(\gamma_k+c\sigma^2)^2}
\label{eq:NORM-d2}
\end{align}
almost surely, using the results of Lemma \ref{lem:LLN} and \cite{BaikSilverstein06}, Theorem 1.1. Now notice that
\begin{align}
&n\left(\frac{n-p-1}{np}\sum\limits_{c=1}^p\frac{\sigma^2}{l_c}-1\right)
\notag\\&=
n\left[\frac{n-p-1}{np}\sum\limits_{c=1}^p\frac{\sigma^2}{l_c}-\frac{n-p-1}{n(1-\frac{p-\rho}n)}+\frac{n-p-1}{n(1-\frac{p-\rho}n)}-1\right]
\notag\\&=
\left(1-\frac{p}n-\frac1n\right)n\left[\frac{1}{p}\sum\limits_{c=1}^p\frac{\sigma^2}{l_c}-\frac{1}{1-\frac{p-\rho}n}\right]-\frac {n(\rho+1)}{n-p}\label{eq:NORM-c1}
\end{align}
and
\begin{align}
&n\left(\frac1{1-\frac pn}-\frac{(n-p-1)(n-p-2)}{n^2p}\sum\limits_{c=1}^{p}\frac{\sigma^4}{l_c^2}\right)
\notag\\&\qquad=
n\left(\frac{n}{n-p}-\frac{(n-p-1)(n-p-2)}{n^2(1-\frac{p-\rho}n)^3}
\right.\notag\\&\qquad\qquad\left.
+\frac{(n-p-1)(n-p-2)}{n^2(1-\frac{p-\rho}n)^3}-\frac{(n-p-1)(n-p-2)}{n^2p}\sum\limits_{c=1}^{p}\frac{\sigma^4}{l_c^2}\right)
\notag\\&\qquad=
\left(1-\frac{p}n-\frac1n\right)\left(1-\frac{p}n-\frac2n\right)n\left[\frac1{(1-\frac{p-\rho}n)^3}-\frac1{p}\sum\limits_{c=1}^{p}\frac{\sigma^4}{l_c^2}\right]
\notag\\&\qquad\qquad
+\frac{\rho n^2}{(n-p)(n-p+\rho)}+\frac{(2r+3)n^2}{(n-p+\rho)^2}-\frac{(\rho^2+3r+2)n^2}{(n-p+\rho)^3}.\label{eq:NORM-c2}
\end{align}
Moreover,
\begin{align}
&n\left(\frac{(n-p-1)p}{n^2}\left(\frac1p\sum\limits_{c=1}^p\frac{\sigma^2}{l_c}\right)^2-\frac{\frac{p-\rho}n}{1-\frac{p-\rho}n}\right)
\notag\\&\qquad=
n\bigg(\frac{(n-p-1)p}{n^2}\left(\frac1p\sum\limits_{c=1}^p\frac{\sigma^2}{l_c}\right)^2-\frac{(n-p-1)p}{n^2}\frac1{(1-\frac{p-\rho}n)^2}
\notag\\&\qquad\qquad
+\frac{(n-p-1)p}{n^2}\frac1{(1-\frac{p-\rho}n)^2}-\frac{\frac{p-\rho}n}{1-\frac{p-\rho}n}\bigg)
\notag\\&\qquad=
\left(1-\frac{p}n-\frac1n\right)\frac pnn\left[\left(\frac1p\sum\limits_{c=1}^p\frac{\sigma^2}{l_c}\right)^2-\frac1{(1-\frac{p-\rho}n)^2}\right]
\notag\\&\qquad\qquad
+\frac{\rho n}{n-p+\rho}-\frac{(\rho+1)pn}{(n-p+\rho)^2}
\notag\\&\qquad=
\left(1-\frac{p}n-\frac1n\right)\frac pn\left[\frac1p\sum\limits_{c=1}^p\frac{\sigma^2}{l_c}+\frac1{1-\frac{p-\rho}n}\right] \!n\!\left[\frac1p\sum\limits_{c=1}^p\frac{\sigma^2}{l_c}-\frac1{1-\frac{p-\rho}n}\right]
\notag\\&\qquad\qquad
+\frac{\rho n}{n-p+\rho}-\frac{(\rho+1)pn}{(n-p+\rho)^2}.\label{eq:NORM-c3}
\end{align}
Now, divide the sample covariance matrix as
\begin{align*}
S_n=\sigma^2\left(
{\renewcommand{\arraystretch}{1.5}
\begin{array}{c|c}
S^{11}_n & S^{12}_n \\
\hline
S^{21}_n & S^{22}_n
\end{array}}
\right)
\end{align*}
with $S^{11}_n$ $\rho\times\rho$. Then $S^{22}_n$ has a $\text{W}_{p-\rho}(n,I)$ distribution -- let $\mu_1>...>\mu_{p-\rho}$ be its eigenvalues and notice that by Cauchy's interlacing theorem, $l_i>\sigma^2\mu_i>l_{i+\rho}$ for all $i=1,...,p-\rho$. Therefore, we have
\begin{align}
&
\frac{p-\rho}pn\left[\frac1{p-\rho}\sum\limits_{c=1}^{p-\rho}\frac{1}{\mu_c}-\frac1{1-\frac{p-\rho}n}\right]+\rho\frac np\frac{\sigma^2}{l_p}-\frac{\rho n^2}{(n-p+\rho)p}
\notag\\&\qquad\qquad\geq\;
n\left[\frac1p\sum\limits_{c=1}^p\frac{\sigma^2}{l_c}-\frac1{1-\frac{p-\rho}n}\right]
\geq\;
\notag\\&
\frac{p-\rho}pn\left[\frac1{p-\rho}\sum\limits_{c=1}^{p-\rho}\frac{1}{\mu_c}-\frac1{1-\frac{p-\rho}n}\right]+\frac np\sum\limits_{c=1}^{\rho}\frac{\sigma^2}{l_c}-\frac{\rho n^2}{(n-p+\rho)p}\label{eq:NORM-b1}
\end{align}
and
\begin{align}
&
\frac{p-\rho}pn\left[\frac1{p-\rho}\sum\limits_{c=1}^{p-\rho}\frac{1}{\mu^2_c}-\frac1{(1-\frac{p-\rho}n)^3}\right]+\rho\frac np\frac{\sigma^4}{l_p^2}-\frac{\rho n^4}{(n-p+\rho)^3p}
\notag\\&\qquad\qquad\geq\;
n\left[\frac1p\sum\limits_{c=1}^p\frac{\sigma^4}{l^2_c}-\frac1{(1-\frac{p-\rho}n)^3}\right]
\geq\;
\notag\\&
\frac{p-\rho}pn\left[\frac1{p-\rho}\sum\limits_{c=1}^{p-\rho}\frac{1}{\mu^2_c}-\frac1{(1-\frac{p-\rho}n)^3}\right]+\frac np\sum\limits_{c=1}^{\rho}\frac{\sigma^4}{l_c^2}-\frac{\rho n^4}{(n-p+\rho)^3p}.\label{eq:NORM-b2}
\end{align}
Consequently, let us study the quantities
\begin{align*}
n\left[\frac1{p-\rho}\sum\limits_{c=1}^{p-\rho}\frac1{\mu_c}-\frac1{1-\frac{p-\rho}n}\right]
\quad\text{and}\quad
n\left[\frac1{p-\rho}\sum\limits_{c=1}^{p-\rho}\frac{1}{\mu_c^2}-\frac1{(1-\frac{p-\rho}n)^3}\right].
\end{align*}
We use Theorem 1.1 in \cite{BaiSilverstein04}. First notice that, in our white Wishart case, what they write $F^{c,H}$ is the c.d.f. of a Mar\v cenko-Pastur distribution with parameter $c$. Let $c_n=(p-\rho)/n$ - then
\begin{align*}
&\int \frac1x dG_n(x)=n\left[\int\frac1x dF^{S^{22}_n}(x)-\int\frac1x dF^{c_n,F^I}(x)\right]
\\&\qquad=
n\left[\frac1{p-\rho}\sum\limits_{c=1}^{p-\rho}\frac1{\mu_c}
-\bigints_{[1-\sqrt{c_n}]^2}^{[1+\sqrt{c_n}]^2}\hspace{-40pt}
\frac{\sqrt{([1+\sqrt{c_n}]^2-t)(t-[1-\sqrt{c_n}]^2)}}{2\pi c_nt^2}dt\right]
\\&\qquad=
n\left[\frac1{p-\rho}\sum\limits_{c=1}^{p-\rho}\frac1{\mu_c}-\frac1{(1-\frac{p-\rho}n)^2}\right]
\end{align*}
and
\begin{align*}
&\int \frac1{x^2} dG_n(x)=n\left[\int\frac1{x^2} dF^{S^{22}}(x)-\int\frac1{x^2} dF^{c_n,F^I}(x)\right]
\\&\qquad=
n\left[\frac1{p-\rho}\sum\limits_{c=1}^{p-\rho}\frac1{\mu^2_c}
-\bigints_{[1-\sqrt{c_n}]^2}^{[1+\sqrt{c_n}]^2}\hspace{-40pt}
\frac{\sqrt{([1+\sqrt{c_n}]^2-t)(t-[1-\sqrt{c_n}]^2)}}{2\pi c_nt^3}dt\right]
\\&\qquad=
n\left[\frac1{p-\rho}\sum\limits_{c=1}^{p-\rho}\frac1{\mu^2_c}-\frac1{(1-\frac{p-\rho}n)^3}\right].
\end{align*}
But according to the theorem, as $p_n/n\rightarrow c\in(0,1)$
\begin{align*}
\left(\int \frac1{x} dG_n(x),\int \frac1{x^2} dG_n(x)\right)
\xrightarrow[n\rightarrow\infty]{\mathcal{D}}
\text{N}_2\left(\tilde\mu,\tilde\Sigma\right)
\end{align*}
where the components of $\tilde\mu$ and $\tilde\Sigma$ are given by eq. (1.6) and (1.7) in the theorem statement (\cite{BaiSilverstein04} p.\ 558). To compute these, we follow the arguments of Section 5 from the same paper. According to eq. (5.13) from p.\ 598, we find
\begin{align*}
&\tilde\mu_1=\frac14\left[\frac1{[1-\sqrt{c}]^2}+\frac1{[1+\sqrt{c}]^2}\right]
-\frac1{2\pi}\int_{[1-\sqrt{c}]^2}^{[1+\sqrt{c}]^2}\frac{dt}{t\sqrt{4c-(t-1-c)^2}},
\\&
\tilde\mu_2=\frac14\left[\frac1{[1-\sqrt{c}]^4}+\frac1{[1+\sqrt{c}]^4}\right]
-\frac1{2\pi}\int_{[1-\sqrt{c}]^2}^{[1+\sqrt{c}]^2}\frac{dt}{t^2\sqrt{4c-(t-1-c)^2}},
\end{align*}
and the integrals give, after a Poisson substitution $t=1+c-2\sqrt{c}\cos\theta$,
\begin{align*}
&\frac1{2\pi}\int_{[1-\sqrt{c}]^2}^{[1+\sqrt{c}]^2}\frac{dt}{t\sqrt{4c-(t-1-c)^2}}
=
\frac1{2\pi}\int_{0}^{\pi}\frac{d\theta}{1+c-2\sqrt{c}\cos\theta}
\\&\qquad=
\frac1{2\pi}\left[\frac2{1-c}\arctan\left(\frac{1+\sqrt{c}}{1-\sqrt{c}}\tan(\theta/2)\right)\right]_{0}^{\pi}
\\&\qquad=
\frac1{2\pi}\left[\frac2{1-c}\frac{\pi}2-0\right]=\frac1{2(1-c)},
\\&
\frac1{2\pi}\int_{[1-\sqrt{c}]^2}^{[1+\sqrt{c}]^2}\frac{dt}{t^2\sqrt{4c-(t-1-c)^2}}
=
\frac1{2\pi}\int_{0}^{\pi}\frac{d\theta}{(1+c-2\sqrt{c}\cos\theta)^2}
\\&\qquad=
\frac1{2\pi}\left[\frac{2(1+c)}{(1-c)^3}\arctan\left(\frac{1+\sqrt{c}}{1-\sqrt{c}}\tan(\theta/2)\right)
\right.\\&\qquad\qquad\qquad\left.
+\frac1{(1-c)^2}\frac{2\sqrt{c}\sin\theta}{1+c-2\sqrt{c}\cos\theta}\right]_{0}^{\pi}
\\&\qquad=
\frac1{2\pi}\left[\frac{2(1+c)}{(1-c)^3}\frac{\pi}2+0-0-0\right]
=\frac{1+c}{2(1-c)^3}.
\end{align*}
Therefore, we obtain
\begin{align*}
&\tilde\mu_1
=\frac{1+c}{2[1-\sqrt{c}]^2[1+\sqrt{c}]^2}-\frac1{2(1-c)}
=\frac{c}{(1-c)^2},
\\&
\tilde\mu_2=\frac{1+6c+c^2}{2[1-\sqrt{c}]^4[1+\sqrt{c}]^4}
-\frac{1+c}{2(1-c)^3}
=\frac{c(c+3)}{(1-c)^4}.
\end{align*}
For the variances, according to (1.16), p.\ 564, we can write
\begin{align*}
&\tilde\Sigma_{11}=-\frac1{2\pi^2}\oint_{C_1}\oint_{C_2}\frac{dm_1dm_2}{z(m_1)z(m_2)(m_1-m_2)^2},
\\&\tilde\Sigma_{12}=-\frac1{2\pi^2}\oint_{C_1}\oint_{C_2}\frac{dm_1dm_2}{z(m_1)^2z(m_2)(m_1-m_2)^2},
\\&\tilde\Sigma_{22}=-\frac1{2\pi^2}\oint_{C_1}\oint_{C_2}\frac{dm_1dm_2}{z(m_1)^2z(m_2)^2(m_1-m_2)^2},
\end{align*}
where $C_1$, $C_2$ are contours that can be chosen counterclockwise, nonintersecting and enclosing $1/(c-1)$ (cf. p.\ 598), and where $z(m)$ stands for the inverse Stieltjes transform of the complimentary Mar\v cenko-Pastur distribution, which has closed form
\begin{align*}
z(m)=-\frac1{m}+\frac{c}{1+m}.
\end{align*}
We first find:
\begin{align*}
&\oint_{C_1}\frac{dm_1}{z(m_1)(m_1-m_2)^2}
=-\frac1{1-c}\oint_{C_1}\frac{m_1(m_1+1)}{(m_1-m_2)^2}\frac{dm_1}{m_1-1/(c-1)}
\\&\qquad
=-\frac{2\pi c i }{(1-c)^3}\frac{1}{[m_2-1/(c-1)]^2},
\\&\oint_{C_1}\frac{dm_1}{z(m_1)^2(m_1-m_2)^2}
=\frac1{(1-c)^2}\oint_{C_1}\frac{m_1^2(m_1+1)^2}{(m_1-m_2)^2}\frac{dm_1}{[m_1-1/(c-1)]^2}
\\&\qquad
=-4\pi i\frac{c}{(1-c)^5}\frac{1}{[m_2-1/(c-1)]^2}.
\end{align*}
But then,
\begin{align*}
&\oint_{C_2}\frac{dm_2}{z(m_2)[m_2-1/(c-1)]^2}
=
-\frac1{1-c}\oint_{C_2}\frac{m_2(m_2+1)dm_2}{[m_2-1/(c-1)]^3}
\\&\qquad
=-2\pi i\frac2{2!(1-c)}
=-2\pi i\frac{1}{1-c}
\\&\oint_{C_2}\frac{dm_2}{z(m_2)^2[m_2-1/(c-1)]^2}
=
\frac1{(1-c)^2}\oint_{C_2}\frac{m^2_2(m_2+1)^2dm_2}{[m_2-1/(c-1)]^4}
\\&\qquad
=2\pi i\frac{12(1-2/(1-c))}{3!(1-c)^2}
=-4\pi i\frac{1+c}{(1-c)^3}.
\end{align*}
Therefore,
\begin{align*}
&\tilde\Sigma_{11}=\frac{2c}{(1-c)^4},
\quad
\tilde\Sigma_{12}=\frac{4c}{(1-c)^6}
\quad\text{ and }\quad
\tilde\Sigma_{22}=\frac{8c(1+c)}{(1-c)^8}.
\end{align*}
In summary,
\begin{align}
&\left(\int \frac1{x} dG_n(x),\int \frac1{x^2} dG_n(x)\right)
\notag\\&\qquad\qquad
\xrightarrow[n\rightarrow\infty]{\mathcal{D}}
\text{N}_2\left(
\left[\begin{array}{c}
\frac{c}{(1-c)^2} \\
\frac{c(c+3)}{(1-c)^4}
\end{array}\right]
,
\left[\begin{array}{cc}
\frac{2c}{(1-c)^4} & \frac{4c}{(1-c)^6}\\
\frac{4c}{(1-c)^6} & \frac{8c(1+c)}{(1-c)^8}
\end{array}\right]
\right).\label{eq:NORM-l1}
\end{align}
Therefore, going back to eq.\ (\ref{eq:NORM-a1}) and ineq.\ (\ref{eq:NORM-c1}),  (\ref{eq:NORM-c2}), (\ref{eq:NORM-c3}), (\ref{eq:NORM-b1}) and (\ref{eq:NORM-b2}), we have the upper bound
{\setlength{\mathindent}{0pt}\begin{align*}&
n(\tilde\sigma^2-\sigma^2)
\notag\\&\quad=
\frac{\sigma^2}{\sigma^4 B}\left[
\left(1-\frac{p}n-\frac1n\right)n\left[\frac{1}{p}\sum\limits_{c=1}^p\frac{\sigma^2}{l_c}-\frac{1}{1-\frac{p-\rho}n}\right]
\right.\notag\\&\qquad\qquad\qquad\left.
-\frac{n(\rho+1)}{n-p}
\right.\notag\\&\qquad\qquad\left.
-\left(1-\frac{p}n-\frac1n\right)\left(1-\frac{p}n-\frac2n\right)n\left[\frac1{p}\sum\limits_{c=1}^{p}\frac{\sigma^4}{l_c^2}-\frac1{(1-\frac{p-\rho}n)^3}\right]
\right.\notag\\&\qquad\qquad\qquad\left.
+\frac{\rho n^2}{(n-p)(n-p+\rho)}+\frac{(2\rho+3)n^2}{(n-p+\rho)^2}-\frac{(\rho^2+3\rho+2)n^2}{(n-p+\rho)^3}
\right.\notag\\&\qquad\qquad\left.
+\left(1-\frac{p}n-\frac1n\right)\frac pn\left[\frac1p\sum\limits_{c=1}^p\frac{\sigma^2}{l_c}+\frac1{1-\frac{p-\rho}n}\right] n\left[\frac1p\sum\limits_{c=1}^p\frac{\sigma^2}{l_c}-\frac1{1-\frac{p-\rho}n}\right]
\right.\notag\\&\qquad\qquad\qquad\left.
+\frac{rn}{n-p+\rho}-\frac{(\rho+1)pn}{(n-p+\rho)^2}
\right]
+\frac{nE_n}{B}
\notag\\&\quad\leq
\frac{\sigma^2}{\sigma^4 B}\left[
\left(1-\frac{p}n-\frac1n\right)\frac{p-\rho}pn\left[\frac1{p-\rho}\sum\limits_{c=1}^{p-\rho}\frac{1}{\mu_c}-\frac1{1-\frac{p-\rho}n}\right]
-\frac{n(\rho+1)}{n-p}
\right.\notag\\&\qquad\qquad\qquad\left.
+\left(1-\frac{p}n-\frac1n\right)\rho\frac np\frac{\sigma^2}{l_p}
-\left(1-\frac{p}n-\frac1n\right)\frac{\rho n^2}{(n-p+\rho)p}
\right.\notag\\&\qquad\qquad\left.
-\left(1-\frac{p}n-\frac1n\right)\!\left(1-\frac{p}n-\frac2n\right)\!\frac{p-\rho}pn\!\left[\frac1{p-\rho}\sum\limits_{c=1}^{p-\rho}\frac{1}{\mu^2_c}-\frac1{(1-\frac{p-\rho}n)^3}\right]
\right.\notag\\&\qquad\qquad\qquad\left.
-\left(1-\frac{p}n-\frac1n\right)\left(1-\frac{p}n-\frac2n\right)\frac np\sum\limits_{c=1}^{\rho}\frac{\sigma^4}{l_c^2}
\right.\notag\\&\qquad\qquad\qquad\left.
+\left(1-\frac{p}n-\frac1n\right)\left(1-\frac{p}n-\frac2n\right)\frac{\rho n^4}{(n-p+\rho)^3p}
\right.\notag\\&\qquad\qquad\qquad\left.
+\frac{\rho n^2}{(n-p)(n-p+\rho)}+\frac{(2\rho+3)n^2}{(n-p+\rho)^2}-\frac{(\rho^2+3\rho+2)n^2}{(n-p+\rho)^3}
\right.\notag\\&\qquad\qquad\left.
+\left(1-\frac{p}n-\frac1n\right)\frac pn\left[\frac1p\sum\limits_{c=1}^p\frac{\sigma^2}{l_c}+\frac1{1-\frac{p-\rho}n}\right]\cdot
\right.\notag\\&\hspace{150pt}\left.
\frac{p-\rho}pn\left[\frac1{p-\rho}\sum\limits_{c=1}^{p-\rho}\frac{1}{\mu_c}-\frac1{1-\frac{p-\rho}n}\right]
\right.\notag\\&\qquad\qquad\qquad\left.
+\left(1-\frac{p}n-\frac1n\right)\frac pn\left[\frac1p\sum\limits_{c=1}^p\frac{\sigma^2}{l_c}+\frac1{1-\frac{p-\rho}n}\right]
\rho\frac np\frac{\sigma^2}{l_p}
\right.\notag\\&\qquad\qquad\qquad\left.
-\left(1-\frac{p}n-\frac1n\right)\frac pn\left[\frac1p\sum\limits_{c=1}^p\frac{\sigma^2}{l_c}+\frac1{1-\frac{p-\rho}n}\right]
\frac{\rho n^2}{(n-p+\rho)p}
\right.\notag\\&\qquad\qquad\qquad\left.
+\frac{\rho n}{n-p+\rho}-\frac{(\rho+1)pn}{(n-p+\rho)^2}
\right]
+\frac{nE_n}{B}
\notag\\&\quad=
\frac{\sigma^2}{\sigma^4 B}\left(a^{(n)}_1n\left[\frac1{p-\rho}\sum\limits_{c=1}^{p-\rho}\frac{1}{\mu_c}-\frac1{1-c}\right]+a^{(n)}_2n\left[\frac1{p-\rho}\sum\limits_{c=1}^{p-\rho}\frac{1}{\mu^2_c}-\frac1{(1-c)^3}\right]
\right.\notag\\&\hspace{100pt}\left.
+b^{(n)}\right)+\frac{nE_n}{B}.
\end{align*}}
Now note that 
{\setlength{\mathindent}{5pt}\begin{align*}&
\qquad
a^{(n)}_1\overset{\text{a.s.}}{\rightarrow}1+c,
\qquad
a^{(n)}_2\overset{\text{a.s.}}{\rightarrow}-(1-c)^2,
\\&
b^{(n)}\overset{\text{a.s.}}{\rightarrow}
\frac{2c(\rho+1)-1}{(1-c)^2}+\!\frac{(2-c)(1+c)\rho}{(1-\sqrt{c})^2}
-\!\frac{(1-c)^2}c\!\sum_{k=1}^\rho\frac{\sigma^4\gamma_k^2}{(\gamma_k+\sigma^2)^2(\gamma_k+c\sigma^2)^2}
\end{align*}}
using Lemma \ref{lem:LLN} and \cite{BaikSilverstein06}, Theorem 1.1.
Therefore, using Slutsky and eq. (\ref{eq:NORM-l1}),
\begin{align*}&
a^{(n)}_1n\left[\frac1{p-\rho}\sum\limits_{c=1}^{p-\rho}\frac{1}{\mu_c}-\frac1{1-c}\right]
+a^{(n)}_2n\left[\frac1{p-\rho}\sum\limits_{c=1}^{p-\rho}\frac{1}{\mu^2_c}-\frac1{(1-c)^3}\right]+b^{(n)}
\\&\qquad\qquad
\overset{\mathcal{D}}{\longrightarrow}\quad
\text{N}\left(\mu,\frac{2c(1+c)^2}{(1-c)^4}\right),
\end{align*}
with 
$$\mu=\frac{(2-c)(1+c)\rho}{(1-\sqrt{c})^2}+\frac{2c\rho-1}{(1-c)^2}-\frac{(1-c)^2}c\sum_{k=1}^\rho\frac{\sigma^4\gamma_k^2}{(\gamma_k+\sigma^2)^2(\gamma_k+c\sigma^2)^2}.$$
Therefore, using eq. (\ref{eq:NORM-d1}) and (\ref{eq:NORM-d2}) we obtain $n(\tilde\sigma^2-\sigma^2)\leq X^+_n$ with $X^+_n\xrightarrow[n\rightarrow\infty]{\mathcal{D}}\text{N}\left(\mu^+,\frac{2c(1+c)^2\sigma^4}{(1-c)^4}\right)$, where
{\setlength{\mathindent}{5pt}\begin{align}&
\mu^+=\frac{(2c\rho-1)\sigma^2}{(1-c)^2}+\!\frac{(2-c)(1+c)\rho\sigma^2}{(1-\sqrt{c})^2}-\!\frac{(1-c)^2}c\!\sum_{k=1}^\rho\frac{\sigma^6\gamma_k^2}{(\gamma_k+\sigma^2)^2(\gamma_k+c\sigma^2)^2}
\notag\\&\qquad\qquad
+\frac{(1-c)^2}c\sum_{k=1}^\rho\frac{\gamma_k^2\bar\gamma_k\sigma^4}{(\gamma_k+\sigma^2)^2(\gamma_k+c\sigma^2)^2}
+\sum_{k=1}^\rho\frac{\gamma_k\bar\gamma_k\sigma^2}{(\gamma_k+\sigma^2)(\gamma_k+c\sigma^2)}
\notag\\&\qquad\qquad
-2\sum_{k=1}^\rho\frac{\gamma_k\bar\gamma_k\sigma^2}{(\gamma_k+c\sigma^2)^2}.
\label{eq:NORM-mu+}
\end{align}}
The same argument can be done to obtain a lower bound. Using eq. (\ref{eq:NORM-a1}), (\ref{eq:NORM-c1}),  (\ref{eq:NORM-c2}), (\ref{eq:NORM-c3}), (\ref{eq:NORM-b1}) and (\ref{eq:NORM-b2}) again, we get
\begin{align*}&
n(\tilde\sigma^2-\sigma^2)
\leq\;
\frac{\sigma^2}{\sigma^4 B}\left(
a^{(n)}_1n\left[\frac1{p-\rho}\sum\limits_{c=1}^{p-\rho}\frac{1}{\mu_c}-\frac1{1-c}\right]
\right.\notag\\&\hspace{50pt}\left.
+a^{(n)}_2n\left[\frac1{p-\rho}\sum\limits_{c=1}^{p-\rho}\frac{1}{\mu^2_c}-\frac1{(1-c)^3}\right]
+b^{(n)}
\right)+\frac{nE_n}{B}
\end{align*}
where
{\setlength{\mathindent}{5pt}\begin{align*}&
a^{(n)}_1\overset{\text{a.s.}}{\rightarrow}1+c,
\qquad
a^{(n)}_2\overset{\text{a.s.}}{\rightarrow}-(1-c)^2,
\\&
b^{(n)}\overset{\text{a.s.}}{\rightarrow}
\frac{2c(\rho+1)-1}{(1-c)^2}-\frac{(1-c)^2\rho}{c(1-\sqrt{c})^2}
+\frac{1+c}{c}\sum_{k=1}^\rho\frac{\sigma^2\gamma_k}{(\gamma_k+\sigma^2)(\gamma_k+c\sigma^2)}
\end{align*}}
Therefore, again using eq. (\ref{eq:NORM-d1}) and (\ref{eq:NORM-d2}) we obtain that $n(\tilde\sigma^2-\sigma^2)\geq X^-_n$ with $X^-_n\xrightarrow[n\rightarrow\infty]{\mathcal{D}}\text{N}\left(\mu^-,\frac{2c(1+c)^2\sigma^4}{(1-c)^4}\right)$, where
{\setlength{\mathindent}{5pt}\begin{align}&
\mu^-
=
\frac{(2c\rho-1)\sigma^2}{(1-c)^2}-\frac{(1-c)^2\rho\sigma^2}{c(1-\sqrt{c})^2}
+\frac{1+c}{c}\sum_{k=1}^\rho\frac{\sigma^4\gamma_k}{(\gamma_k+\sigma^2)(\gamma_k+c\sigma^2)}
\notag\\&\qquad\qquad
+\frac{(1-c)^2}c\sum_{k=1}^\rho\frac{\gamma_k^2\bar\gamma_k\sigma^4}{(\gamma_k+\sigma^2)^2(\gamma_k+c\sigma^2)^2}
+\sum_{k=1}^\rho\frac{\gamma_k\bar\gamma_k\sigma^2}{(\gamma_k+\sigma^2)(\gamma_k+c\sigma^2)}
\notag\\&\qquad\qquad
-2\sum_{k=1}^\rho\frac{\gamma_k\bar\gamma_k\sigma^2}{(\gamma_k+c\sigma^2)^2}
-3c\sum_{k=1}^\rho\frac{\bar\gamma_k\sigma^4}{(\gamma_k+c\sigma^2)^2}.
\label{eq:NORM-mu-}
\end{align}}
This concludes the proof.
\end{proof}

\begin{proof}[\bf Proof of Lemma \ref{lem:TV}]
Let $\Sigma'\in\text{B}_r(\Sigma,2M)$, and write $\lambda_i=\lambda_i(\Sigma)$, $\lambda'_i=\lambda'_i(\Sigma'_p)$ to simplify notation. Since the sequences are spiked, there are a finite number of different eigenvalues as $n\rightarrow\infty$. We can decompose
\begin{align*}
&\text{N}(0,\Sigma_p)^n=\bigotimes_{i=1}^{p}\text{N}(0,\lambda_i)^n,
\qquad
\text{N}(0,\Sigma'_p)^n=\bigotimes_{i=1}^{p}\text{N}(0,\lambda'_i)^n.
\end{align*}
Let $\lambda$ be the Lebesgue measure on $\mathbb{R}$. Writing the Hellinger affinity between two densities $f,g$ as $\alpha(f,g)=\int\sqrt{fg}d\lambda$, we have by Cauchy-Schwarz
\begin{align*}
\delta_{TV}\big(f,g\big)^2=4(\int |f-g|d\lambda)^2&\leq\; \frac14\big(1-\alpha(f,g)\big)\big(1+\alpha(f,g)\big)
\\&=\;(1-\alpha(f,g)^2)
\end{align*}
Now for two normals, we have
\begin{align*}
\alpha\big(\text{N}(0,a^2),\text{N}(0,b^2)\big)^2=\frac{2ab}{a^2+b^2}
\end{align*}
and since Hellinger affinity distributes over a product of independent densities, we get
\begin{align*}
&\delta_{TV}\left(\vphantom{\bigg\vert}\text{N}(0,\Sigma_p)^n,\text{N}(0,\Sigma'_{p})^n\right)^2
\leq\;1-
\prod_{i=1}^p\left(\frac{2\sqrt{\lambda_i\lambda'_i}}{\lambda_i+\lambda'_i}\right)^n
\end{align*}
Now note that $\frac{2\sqrt{\lambda_i\lambda'_i}}{\lambda_i+\lambda'_i}>\frac{\sqrt{1-2M/n^r}}{1-M/n^r}$ if and only if $\lambda_i(1-\frac{2M}{n^r})<x<\lambda_i(1+\frac{2M}{n^r(1-2M/n^r)})$. By definition, $\Sigma'\in\text{B}_r(\Sigma,2M)$ means $|\lambda_i-\lambda_i'|<\frac{2M\lambda_i}{n^r}$ for all $i$, so we have the bound
\begin{align*}
&\qquad<\;
1-
\left(\frac{\sqrt{1-2M/n^r}}{1-M/n^r}\right)^{np}
=\;
1-
\left(1-\frac{M^2}{n^{2r}(1-M/n^r)^2}\right)^{np/2}
\notag
\end{align*}
over all $\Sigma'\in\text{B}_r(\Sigma,2M)$. Taking a supremum and a limit yields
\begin{align*}&
\lim_{n\rightarrow\infty}\sup_{\Sigma'\in\text{B}_1(\Sigma,2M)}\delta_{TV}\left(\vphantom{\bigg\vert}\text{N}(0,\Sigma_p)^n,\text{N}(0,\Sigma'_{p})^n\right)^2
\\&\qquad\leq\;
\lim_{n\rightarrow\infty}
\left[1-
\left(1-\frac{M^2}{n^2(1-M/n)^2}\right)^{np/2}
\right]
=\;
1-
e^{-cM^2/2}
\end{align*}
and
\begin{align*}&
\lim_{n\rightarrow\infty}\sup_{\Sigma'\in\text{B}_r(\Sigma,2M)}\delta_{TV}\left(\vphantom{\bigg\vert}\text{N}(0,\Sigma_p)^n,\text{N}(0,\Sigma'_{p})^n\right)^2=0,
\end{align*}
as desired.
\end{proof}

\begin{proof}[\bf Proof of Theorem \ref{thm:MRLB}]
Define $\Sigma'_p=(1-\frac{2M_\epsilon}n)\Sigma_p$. We have 
\begin{align}
|\sigma^2-\sigma^{2\prime}|= 2\sigma^2\frac{M_\epsilon}n.\label{eq:thmMRLB-1}
\end{align}
Moreover, according to Lemma \ref{lem:TV}, we have
\begin{align*}
\lim_{n\rightarrow\infty}\delta_{TV}\left(\vphantom{\bigg\vert}\text{N}(0,\Sigma_p)^n,\text{N}(0,\Sigma'_{p,M})^n\right)\leq \sqrt{1-\exp\left(-\frac{cM_\epsilon^2}{2}\right)}
=1-4\epsilon,
\end{align*}
so for some $N_\epsilon$ we have $\delta_{TV}<1-2\epsilon$ for all $n\geq N_\epsilon$.
Now note that $\Sigma'_p\in\text{B}_1(\Sigma_p,\frac{2M_\epsilon}n\|\Sigma_p\|_2) $. Say for some $\hat\sigma^2$ we have
\begin{align}
\sup_{\substack{\Sigma'\in\text{B}_1(\Sigma,2M_\epsilon)}}
\text{P}_{\Sigma'_p}\left[\vphantom{\Bigg\vert}|\hat\sigma^2-\sigma^{2\prime}|\geq \sigma^2\frac{M_\epsilon}n\right]<\epsilon.\label{eq:thmMRLB-2}
\end{align}
Define the event $A=\left[|\hat\sigma^2-\sigma^2|\leq \sigma^2\frac{M_\epsilon}n\right]$ - then by (\ref{eq:thmMRLB-1}) and (\ref{eq:thmMRLB-2}) we find $\text{P}_{\Sigma_p}[A]\geq 1-\epsilon$ and $\text{P}_{\Sigma'_p}[A]<\epsilon$. Therefore,
\begin{align*}
\delta_{TV}\left(\vphantom{\bigg\vert}\text{N}(0,\Sigma_p)^n,\text{N}(0,\Sigma'_{p,M})^n\right)\geq \text{P}_{\Sigma_p}[A]-\text{P}_{\Sigma'_p}[A]\geq 1-2\epsilon
\end{align*}
which contradicts $\delta_{TV}<1-2\epsilon$ for all $n\geq N_\epsilon$ - therefore,
\begin{align*}
\sup_{\substack{\Sigma'\in\text{B}_1(\Sigma,2M_\epsilon)}}
\text{P}_{\Sigma'_p}\left[\vphantom{\Bigg\vert}|\hat\sigma^2-\sigma^{2\prime}|\geq \sigma^2\frac{M_\epsilon}{n^r}\right]\geq\epsilon
\end{align*}
for all $\hat\sigma^2$ and $n\geq N_\epsilon$, yielding the desired result.
\end{proof}

\begin{proof}[\bf Proof of Proposition \ref{prop:RATE}]
We have:
\begin{align*}&
\text{P}_{\Sigma'_p}\left[|\tilde\sigma^2-\sigma^{2\prime}|\geq \sigma^2\frac{M}{n^r}\right]
\leq
\text{P}_{\Sigma_p}\left[|\tilde\sigma^2-\sigma^{2\prime}|\geq \sigma^2\frac{M}{n^r}\right]
\\&\qquad
+\bigg|\text{P}_{\Sigma_p}\left[|\tilde\sigma^2-\sigma^{2\prime}|\geq \sigma^2\frac{M}{n^r}\right]-\text{P}_{\Sigma'_p}\left[|\tilde\sigma^2-\sigma^{2\prime}|\geq \sigma^2\frac{M}{n^r}\right]\bigg|
\\&\quad= A_p+B_p.
\end{align*}
Choose any $\Sigma'\in\text{B}_r(\Sigma,2M)$ and write $\lambda_i=\lambda_i(\Sigma)$, $\lambda'_i=\lambda'_i(\Sigma'_p)$. Using Lemma \ref{lem:TV}, we find for the second term
\begin{align*}
\lim_{n\rightarrow\infty} \sup_{\Sigma\in\text{B}_r(\Sigma,2M)}B_p
&\leq \lim_{n\rightarrow\infty}\sup_{\Sigma'\in\text{B}_r(\Sigma,2M)}\delta_{TV}\left(\vphantom{\bigg\vert}\text{N}(0,\Sigma_p)^n,\text{N}(0,\Sigma'_p)^n\right)
\\&=0.
\end{align*}
For the first term, let us use Theorem \ref{thm:NORM}. Since $|\sigma^{2\prime}-\sigma^2|<\sigma^2M/n^r$, we have
{\setlength{\mathindent}{5pt}\begin{align*}&
\sup_{\Sigma\in\text{B}(\Sigma,2M)}A_p\leq\;
\text{P}_{\Sigma_p}\left[n\left|\tilde\sigma^2-\sigma^{2}\right|>
	\sigma^2\frac{M}{n^{r-1}}\right]
\\&\qquad\leq
\text{P}_{\Sigma_p}\left[\frac{(1-c)^2}{\sqrt{2c}(1+c)\sigma^2}\left(X_n^+-\mu^+\right)>
	\frac{(1-c)^2(-2\sigma^2\mu^++M/n^{r-1})}{2\sqrt{2c}(1+c)}\right]
\\&\qquad\quad+
\text{P}_{\Sigma_p}\left[\frac{(1-c)^2}{\sqrt{2c}(1+c)\sigma^2}\left(X_n^--\mu^-\right)<
	\frac{(1-c)^2(-2\sigma^2\mu^--M/n^{r-1})}{2\sqrt{2c}(1+c)}\right]
\\&\qquad\xrightarrow{n\rightarrow\infty}1-\Phi(\infty)+\Phi(-\infty)=0.
\end{align*}}
This concludes the proof.
\end{proof}

\subsection{Proofs for Section \ref{sec:A}}

\begin{proof}[\bf Proof of proposition \ref{prop:CONS}]
We first notice that, from Lemma \ref{lem:LLN} and Theorem \ref{thm:NORM}, for any $1\leq k\leq \rho$ we have
\begin{align}
\hat\gamma_{\rho,k}\xrightarrow[n\rightarrow\infty]{\text{a.s.}}\gamma_k,
\qquad
\tilde\sigma^2_{\rho}\xrightarrow[n\rightarrow\infty]{\text{a.s.}}\sigma^2.
\label{eq:SIM-a1}
\end{align}
Denote as usual the unbiased risk estimator of Theorem \ref{thm:URE} at $\hat\Gamma_r+\tilde\sigma^2_rI_p$ by $F_r+G_r$. By their definitions and Lemma \ref{lem:LLN}, we see that
\begin{align*}
F_{\rho}+G_{\rho}\xrightarrow[n\rightarrow\infty]{\text{a.s.}} 0.
\end{align*}
Thus there exists a $N_1$ (random) such that  $\rho\in \left\{r\,\vert\,\left|F_r+G_r\right|\leq \frac{p+1}n\right\}$ for all $n\geq N_1$. Now note that there is a $N_2$ (also random) such that for all $n\geq N_2$, $\left\{\frac{\1{r<p}}{l_{r+1}}\frac{(1+\sqrt{p/n})^2}{p-r}\sum\limits_{c=r+1}^p l_c\geq 1\right\}=\{\rho\}$. Thus for all $n\geq N_1\vee N_2$, 
\begin{align*}
\hat\rho=\min\{\rho\}=\rho,
\end{align*}
so in particular $\hat\rho\xrightarrow[n\rightarrow\infty]{\text{a.s.}}\rho$. 
Thus, by eq. (\ref{eq:SIM-a1}) and any $\epsilon>0$ there exists a $N_3$ random such that for all $n\geq N_1\vee N_2\vee N_3$, $k<p$ and 
$\left|\hat\gamma_{\hat\rho,k}-\gamma_k\right|=\left|\hat\gamma_{\rho,k}-\gamma_k\right|<\epsilon$.
That is, $\gamma_{\hat\rho,k}\xrightarrow[n\rightarrow\infty]{\text{a.s.}}\gamma_k$. 
Finally, again by eq. (\ref{eq:SIM-a1}) for all $\epsilon>0$, there exists a $N_3$ random such that for all $n\geq N_1\vee N_2\vee N_3$, $|\tilde\sigma^2-\sigma^2|=|\tilde\sigma^2_{\hat\rho}-\sigma^2|=|\tilde\sigma^2_\rho-\sigma^2|<\epsilon$, i.e. $\tilde\sigma^2\xrightarrow[n\rightarrow\infty]{\text{a.s.}}\sigma^2$, as desired.
\end{proof}


\bibliographystyle{plainnat}
\bibliography{biblio}

\end{document}